\documentclass[11pt,A4]{article}
\usepackage{fullpage}
\usepackage{amsmath}
\usepackage{amsfonts}
\usepackage{amssymb}
\usepackage{tikz-cd} 
\usepackage{titlesec}
\usepackage{tasks}
\usepackage{mathabx}
\usepackage{color}
\usepackage{mathrsfs}
\usepackage{cite}
\usepackage[colorlinks,linkcolor=blue]{hyperref}

\title{Level correspondence of $K$-theoretic $I$-function in Grassmann duality}
\author{Hai Dong and Yaoxiong Wen}
\date{}

\newtheorem{Theorem}{Theorem}[section]
\newtheorem{Lemma}{Lemma}[section]
\newtheorem{Proposition}{Proposition}[section]
\newtheorem{Corollary}{Corollary}[section]
\newtheorem{Definition}{Definition}[section]
\newtheorem{Example}{Example}[section] 
\newtheorem{Remark}{Remark}[section]
\newenvironment{proof}{{\noindent\it Proof}\quad}{\hfill $\square$\par}

\begin{document}
\maketitle

\begin{abstract}
In this paper, we prove a class of nontrivial q-Pochhammer symbol identities with extra parameters by iterated residue method. Then we use these identities to find relations of the quasi-map $K$-theoretical $I$-functions with level structure between Grassmannian and its dual Grassmannian. Here we find an interval of levels within which two $I$-functions are the same, and on the boundary of that interval, two $I$-functions are intertwining with each other. We call this phenomenon level correspondence in Grassmann duality.
\end{abstract}

\setcounter{tocdepth}{3}
\setcounter{secnumdepth}{4}

\tableofcontents
\clearpage

\section{Introduction}
\label{intro}
The quantum $K$-theory was introduced by Givental \cite{givental2000wdvv} and Y.P. Lee \cite{lee2004quantum} decades ago. Recently, Givental shows that q-hypergeometric solutions represent $K$-theoretic Gromov-Witten invariants in the toric case \cite{2015arXiv150903903G} and Ruan-Zhang \cite{2018arXiv180406552R} introduce the level structures and there is a serendipitous discovery that some special toric spaces with certain level structures result in Mock theta functions. Nevertheless, beyond the toric case, much less is known.

 The recent explosion of study of the quantum $K$-theory was from a fundamental relation between 3d supersymmetric gauge theories and quantum $K$-theory of so called Higgs branch discovered by the works of Nekrasov \cite{Nekrasov9606} and Nekrasov and Shatashvili \cite{Nekrasov0901} \cite{Nekrasov0908}, amongst many others. For the concrete case of massless theories with a non-trivial UV-IR flow, Hans Jockers and Peter Mayr \cite{jockers20183d} show a 3d gauge theory/quantum $K$-theory correspondence, connecting the BPS partition functions of specific $\mathcal{N}=2$ supersymmetric gauge theories to Givental's permutation equivariant  K-theory. Besides, Hans Jockers, Peter Mayr, Urmi Ninad, Alexander Tabler \cite{jockers2019wilson} and Kazushi Ueda, Yutaka Yoshida \cite{ueda20193d} establish the correspondence between 3d gauge theory and the quantum $K$-ring and $I$-function of Gr(r,n) independently. Now, it is well-understood that the level structures introduced by Ruan-Zhang \cite{2018arXiv180406552R} are the key new feature for so called 3d $\mathcal{N}=2$ theory (Chern-Simons term).

One of key feature of gauge theory is Seiberg-duality which has been studied in 2d by Giulio Bonelli, Antonio Sciarappa, Alessandro Tanzini, Petr Vasko \cite{cite-key} and the first author. As far as authors' knowledge, very little is known in 3d $\mathcal{N}=2$ case. The results of this article hopefully will contribute some clarity. The  simplest example is Grassmannian $Gr(r,V)$ versus dual Grassmannian $Gr(n-r,V^*)$. However, it is unknown how to match the level structure. Without misunderstanding, we will use $Gr(r,n)$ and $Gr(n-r,n)$ to denote Grassmannian and its dual respectively. They are geometrically isomorphic. However, they encode very different combinatorial data. A long-standing problem is to match their combinatorial data directly. For example, the presentations of $K$-theoretic $I$-functions depend on their gauge theory representation/combinatorial data, and it is tough to see why the $I$-function of Grassmannian equals the $I$-function of dual Grassmannian. In this paper, we give the explicit formula of $K$-theoretic $I$-function of Grassmannian with level structure by using abelian/non-abelian correspondence  \cite{2019arXiv190600775W} as follows
\begin{align}
I^{Gr(r,n),E_r,l}_{T,d}=\sum_{d_1+d_2+ \cdots + d_r = d} Q^d \prod_{i,j=1}^r \frac{\prod^{d_i-d_j}_{k=-\infty}(1-q^kL_iL^{-1}_j)}{\prod^{0}_{k=-\infty}(1-q^kL_iL^{-1}_j)}\prod_{i=1}^{r}\frac{(L^{d_{i}}_{i}q^{\frac{d_{i}(d_{i}-1)}{2}})^l}{\prod_{k=1}^{d_i}\prod_{m=1}^{n}(1-q^kL_i\Lambda^{-1}_m)} \nonumber
\end{align}
and
\begin{align}
I^{Gr(n-r,n),E_{n-r},l}_{T,d}=\sum_{d_1+d_2+ \cdots + d_{n-r} = d} Q^d \prod_{i,j=1}^{n-r} \frac{\prod^{d_i-d_j}_{k=-\infty}(1-q^k\tilde{L}_i\tilde{L}^{-1}_j)}{\prod^{0}_{k=-\infty}(1-q^k\tilde{L}_i\tilde{L}^{-1}_j)}\prod_{i=1}^{n-r}\frac{(\tilde{L}^{d_{i}}_{i}q^{\frac{d_{i}(d_{i}-1)}{2}})^l}{\prod_{k=1}^{d_i}\prod_{m=1}^{n}(1-q^k\tilde{L}_i\Lambda _m)} \nonumber	
\end{align}

We want to remark here that the isomorphism between Grassmannian and its dual would imply the equivalence of $J$-function when level $l$ is $0$. In fact, $I$-function is known to be different from $J$-function with negative levels.  

In this paper, we use Proposition \ref{Thm-0} to show the relations of the equivariant $I$-function between Grassmannian $Gr(r,n)$ and that of dual Grassmannian $Gr(n-r,n)$ with level structures, here we find an interval of levels within which two $I$-functions with levels are the same, and on the boundary of that interval, two $I$-functions with levels are intertwining with each other. We call this phenomenon level correspondence in Grassmann duality. The existence of certain interval of level is very mysterious to us. We hope that our result will give some hint how to
formulate Seiberg-duality for a general target. 

\begin{Theorem}(Level Correspondence)
For Grassmannian $Gr(r,n)$ and its dual Grassmannian $Gr(n-r,n)$ with standard $T=(\mathbb{C}^*)^n$ tours action , let $E_r$, $E_{n-r}$ be the standard representation of $\operatorname{GL}(r,\mathbb{C})$ and $\operatorname{GL}(n-r,\mathbb{C})$, respectively. Consider the following equivariant $I$-function
\begin{align*}
	I_T^{Gr(r,n),E_r,l}=&1+\sum_{d=1}^{\infty}I^{Gr(r,n),E_r,l}_{T,d}Q^d  \\
	I_T^{Gr(n-r,n),E_{n-r}^\vee,-l}=&1+\sum_{d=1}^{\infty}I^{Gr(n-r,n),E^\vee_{n-r},-l}_{T,d}Q^d
\end{align*}
Then we have following relations between $I^{Gr(r,n),E_r,l}_{T,d}$ and $I_T^{Gr(n-r,n),E_{n-r}^\vee,-l}$ in $K^{loc}_T(Gr(r,n))\otimes\mathbb{C}(q) \cong K^{loc}_T(Gr(n-r,n))\otimes\mathbb{C}(q)$:

\begin{itemize}
\item For $1-r\leq l\leq n-r-1$, we have 
\begin{align}
			I_{T,d}^{Gr(r,n),E_{r},l} = 	I_{T,d}^{Gr(n-r,n),E_{n-r}^\vee,-l}  \nonumber
\end{align}
\item For $l=n-r$, we have
\begin{align}
	   I_{T,d}^{Gr(r,n),E_{r},l}= \sum_{s=0}^d C_{s}( n-r,d)	I_{T,d-s}^{Gr(n-r,n),E_{n-r}^\vee,-l} \nonumber 
\end{align}	
where $C_{s}(k,d) $ is defined as

\begin{align*}
C_{s}(k,d)=\frac{(-1)^{k s}}{(q; q)_{s} q^{s(d-s+k)} \left(\bigwedge^{top}\mathcal{S}_{n-r} \right)^s  }	
\end{align*}
and $\mathcal{S}_{n-r}$ is the tautological bundle of $Gr(n-r,n)$
\item For $l=-r$, we have
\begin{align}
	 	I_{T,d}^{Gr(n-r,n),E_{n-r}^\vee,-l}= \sum_{s=0}^d D_{s}( r, d)	I_{T,d-s}^{Gr(r,n),E_{r},l} \nonumber
\end{align}
\begin{align*}
D_{s}(r,d)=\frac{(-1)^{rs}}{ (q;q)_sq^{s(d-s)} \left(\bigwedge^{top}\mathcal{S}_{r} \right)^s}
\end{align*}
and $\mathcal{S}_r$ is the tautological bundle of $Gr(r,n)$
\end{itemize}
here we use  q-Pochhammer symbol notation:
\begin{align*}
(a;q)_d : = \left\{ 
\begin{array}{rcl}
& (1-a)(1-qa)\cdots(1-q^{d-1}a)     & d>0 \\
&	1                               & d=0 \\
& \frac{1}{(1-q^{-1}a)\cdots(1-q^{-d}a)}   & d<0 
\end{array}
\right.
\end{align*}
\end{Theorem}

A key step in our proof is the following series of non-trivial $q$-Pochhammer symbol identities which are of independent interest.

Suppose that $r, n, d \in \mathbb{Z}_{>0}, l \in \mathbb{Z}$ and $0<r<n .$ Let $[n]$ be the set of elements $\{1, \ldots, n\}$, $I \subsetneq[n]$ be a $r$-element subset of $[n]$, $I^{\complement}$ be the complementary set of $I$ in $[n]$ and $\vec{d}_{I}$ be $|I|$-tuple of non negative integers. We denote $\sum_{i \in I} d_{i}$ as $\left|\vec{d}_{I}\right|$ and denote $x_{i} / x_{j}$ as $x_{i j}$ for simplicity.
\begin{Proposition}	\label{Thm-0}
	For variables $x_1,\cdots,x_n$, we define the following two expressions involving q-Pochhammer symbols
	\begin{align}
	A_d\left(  \vec{x}, I, l\right)=&\sum_{|\vec{d}_I|=d}\frac{\left( \prod_{i \in I}x_{i}^{d_i}q^{\frac{d_i(d_i-1)}{2}}\right)^l}{\prod_{i ,j \in I}\left( q^{d_{ij}+1}x_{ij};q \right)_{d j}\prod_{i \in I } \prod_{j \in I^\complement } (qx_{ij};q )_{d_i}}   \label{A_d} \\
	B_d\left(  \vec{x}, I, l\right)=&\sum_{|\vec{d}_{I}|=d}\frac{\left(\prod_{i \in I }x_{i}^{-d_i}q^{\frac{d_i(d_i+1)}{2}}\right)^l}{ \prod_{i , j \in I }\left( q^{d_{ij}+1}x_{ji};q \right)_{d j}\prod_{i \in I } \prod_{j \in I^\complement } (qx_{ji} ;q)_{d_i}} \label{B_d}
	\end{align}
If $l$ satisfies the condition
\begin{align}
1-|I| \leq l \leq n-|I|-1 \label{level-condition}	
\end{align}
we have the following identites
\begin{align}
	A_d\left(  \vec{x},I,l\right)= B_d\left(  \vec{x},I^\complement,-l\right)   \nonumber
\end{align}
\end{Proposition}

This paper is arranged as follows. In subsection \ref{sec:2} , we prove Proposition \ref{Thm-0} by constructing a rational function (\ref{f}) and then using iterated residue method which is useful in Nekrosov partition function \cite{felder2018analyticity}. In the following subsection \ref{sec:3}, we provide two explicit examples to explain the proof and also provide a non-trivial identity by using Proposition \ref{Thm-0}. In subsection \ref{sec:4}, we expand the restriction to the boundary, i.e. $l=-|I|$ and $l=n-|I|$. In section \ref{sec:5}, we first revisit $K$-theoretic quasi-map theory in which we review some basic definitions and theorems, especially, the formula of equivariant $I$-function of Grassmannian $Gr(r,n)$, finally in subsection \ref{sec:7}, we apply Proposition \ref{Thm-0} to obtain the level correspondence of $I$-function in Grassmann duality. 

\section*{Acknowledgements}
We would like to thank Prof. Yongbin Ruan for suggesting this problem and useful discussions. Thanks are also due to Prof. Yutaka Yoshida for showing us his formula of $I$-function of Grassmannian and telling us something about Chern-Simons-matter thoery in physics. We would also like to thank Ming Zhang for discussion about level structures, and Zijun Zhou for helpful discussions. We also like to appreciate the hospitality of IASM (Institute For Advanced Study In Mathematics) of Zhejiang University.

\section{The class of q-Pochhammer symbol identities}
\label{sec:1}
\subsection{The proof of identities}
\label{sec:2}
Now we prove this proposition for one case $I=\{1,\cdots,r\}$ by constructing the following symmetric complex rational function $f(w_{1},\cdots ,w_{d})$ with parameters $q$ and $x_1,\cdots ,x_n$, we made following assumptions of parameters
\begin{align}
	&|q|<1 \nonumber
	\\
	x_i x_j^{-1} \neq q^k &\qquad \forall i\neq j \in [n],\forall k\in \mathbb{Z} \label{simPoleCond}
	\end{align}
	furthermore there exists some $\rho>0$ such that 
	\begin{align}
	\max_{i\in [n]} |x_i|<\rho <\min_{i\in [n]} |q|^{-1}|x_i| \label{conditionforq}
	\end{align} 
where $[n]:=\{1, \cdots, n \}$ and general situations follow from analytic continuation. Let $f(w_{1},\cdots ,w_{d})$ be as follows
	\begin{align}
	f(w_{1},\cdots ,w_{d}) &=\frac{1}{(1-q)^d d!}\prod_{i \neq j}^d \frac{w_{i}-w_{j}}{w_{i}-qw_{j}}\prod_{i=1}^d \frac{w_i^{l-1}}{\prod_{j=1}^{r}(1-x_j/w_{i}) \prod_{j=r+1}^{n}(1-qw_{i}/x_{j})} \label{f} \\
	&=g(w_{1},\cdots,w_{d})\prod_{i=1}^{d}\left(\prod_{u \in U}\frac{w_i - q^{-1}u}{w_i - u}\prod_{i<j}\frac{(w_i-w_j)^2}{(w_i-qw_j)(qw_i - w_j)}\right) \label{recursiveFormOfF}
	\end{align}
	where $U$ is a  set of complex numbers all contained in open disk $|w|<\rho$, at the moment $U = \{x_1,\cdots,x_r\}$ and $g$ is a symmetric function of the form
	\begin{align}
	g(\vec{w})=\frac{1}{(1-q)^d d!}\prod_{i=1}^d \frac{w_i^{l+r-1}}{ \prod_{j=1}^{r}(w_i - q^{-1}x_j)\prod_{j=r+1}^{n}(1-qw_{i}/x_{j})} \nonumber
	\end{align}
	from condition (\ref{conditionforq}) and the restriction of $l$ we know $g$ is analytical in the polydiscs $\{(w_1,\cdots,w_n):|w_i| \le \rho, \forall i \in [n]\}$ and $g$ can only have possible zeros for some $w_j = 0$.
	
	We conside the following integration
	\begin{align}
	E_d:=&\int_{C_{\rho}} \frac{dw_{d}}{2\pi\sqrt{-1}}  \ldots \int_{C_{\rho}}  \frac{dw_{1}}{2\pi\sqrt{-1}} f(w_{1},\cdots,w_{d})  \label{integration}
	\end{align}
	where 
	the integration contour $C_{\rho}$ for each variable $w_i$ is the circle centered at origin with radius $\rho$ and takes counter-clockwise direction. The condition (\ref{conditionforq}) ensures that there isn't a pole on the integration contour. By Fubini's theorem, we could permute the order of integration variables and $f(w_{1},\cdots,w_{d})$ is a symmetric function, we can change $(w_{1},\cdots, w_{d})$ to other order.
	
	 In order to get explicit expression of $E_d$, we take take the residues in each variable consecutively inside the integration contour, suppose we have the following evaluating sequence for some $S_1 \le d$ by induction,
	\begin{align}\label{evaluateProcess1}
	\hat{w}_1 = q{w}_2, \hat{w}_2 = q{w}_3, \cdots, \hat{w}_{S_1-1} = q{w}_{S_1}\nonumber
	\end{align}
	which are all simple poles inside $|w|<\rho $, then we have
	\begin{align}
	&\underset{\hat{w}_{S_1-1} = q {w}_{S_1}}{\operatorname{Res}} \cdots \underset{\hat{w}_2 = q {w}_3}{\operatorname{Res}}\underset{\hat{w}_1 = q {w}_2}{\operatorname{Res}}f   \nonumber
	\\=& \prod_{i=S_{1}+1}^{d}\left(\prod_{u \in U}\frac{{w}_i - q^{-1}u}{{w}_i - u}\prod_{i<j}\frac{({w}_i-{w}_j)^2}{({w}_i-q{w}_j)(q{w}_i - {w}_j)}\right) \nonumber
	\\ &\cdot {w}_{S_1}^{S_{1}-1} \prod_{k=0}^{S_{1}-1}\prod_{u \in U}\frac{q^{k}{w}_{S_1} - q^{-1}u}{q^{k}{w}_{S_1} - u} \cdot \prod_{S_1<j}\frac{({w}_{S_1}-{w}_j)(q^{S_{1}-1}{w}_{S_{1}}-{w}_j)}{({w}_{S_1}-q{w}_j)(q^{S_1}{w}_{S_1} - {w}_j)}  \nonumber
	\\
	&\cdot \frac{(q-1)^{S_{1}}}{q^{S_{1}}-1} q^{-({S_{1}-1})(d-{S_{1}})}g(q^{{S_{1}-1}}{w}_{S_{1}},q^{{S_{1}}-2}{w}_{S_{1}},\cdots,{w}_{S_1},{w}_{S_1+1},\cdots,{w}_d)
	\end{align}
	now integrating variable ${w}_{S_1}$, we pick up pole at $\hat{w}_{S_1}  = q^{-k_1}u_1$ for some $0\le k_1< S_1$ and $u_1 \in U=\{x_1,\cdots,x_r\}$. Due to the condition (\ref{conditionforq}),  $|\hat{w}_{S_1}|< \rho$ implies that $k_1=0$. Evaluating at $\hat{w}_{S_1}  = u_1$, we get
	\begin{align}
	& \underset{\hat{w}_{S_1} = u_1}{\operatorname{Res}} \,\underset{\hat{w}_{S_1-1} = q {w}_{S_1}}{\operatorname{Res}} \cdots \underset{\hat{w}_2 = q {w}_3}{\operatorname{Res}} \, \underset{\hat{w}_1 = q {w}_2}{\operatorname{Res}}f   \nonumber
	\\=& \prod_{i=S_{1}+1}^{d}\left( \frac{{w}_i-q^{S_1-1}u_1}{{w}_i-q^{S_1}u_1}\prod_{u \in U\backslash\{u_1\}}\frac{{w}_i - q^{-1}u}{{w}_i - u}\prod_{i<j}\frac{({w}_i-{w}_j)^2}{({w}_i-q{w}_j)(q{w}_i - {w}_j)}\right) \nonumber
	\\ &\cdot u_{1}^{S_1}\prod_{ k =  0}^{S_1 - 1} \prod_{u \in U \backslash \{u_1\} } \frac{q^k u_1-q^{-1}u}{q^k u_1-u}  \cdot  (q-1)^{S_1} q^{S_1 (S_1 -1 - d)} \nonumber 
	\\ &\cdot g(q^{{S_{1}-1}}u_1,q^{{S_{1}}-2}u_1,\cdots,qu_1,u_1,{w}_{S_1+1},\cdots,{w}_d)  \label{firstmethod}
	\\=& \tilde{g}({w}_{S_1+1},\cdots,{w}_{d})\prod_{i=S_1+1}^{d}\left(\prod_{u \in \tilde{U}}\frac{{w}_i - q^{-1}u}{{w}_i - u}\prod_{i<j}\frac{({w}_i-{w}_j)^2}{({w}_i-q{w}_j)(q{w}_i - {w}_j)}\right) \label{remain}
	\end{align}
	where 
	\begin{align}
	\tilde{U} =& U \backslash \{u_1\} \cup \{q^{S_1}u_1\}
	\end{align}
	all elements of $\tilde{U}$ are still all in the open disk $|w|<\rho$, and  
	\begin{align}
	\tilde{g}({w}_{S_1+1},\cdots,{w}_{d}) =& u_{1}^{S_1}\prod_{ k =  0}^{S_1 - 1} \prod_{u \in U \backslash \{u_1\} } \frac{q^k u_1-q^{-1}u}{q^k u_1-u}  \cdot  (q-1)^{S_1} q^{S_1 (S_1 -1 - d)} \nonumber \\ & \cdot g(q^{{S_{1}-1}}u_1,q^{{S_{1}}-2}u_1,\cdots,qu_1,u_1,{w}_{S_1+1},\cdots,{w}_d)
	\end{align}
	so we just write $\tilde{f} :=  \underset{\hat{w}_{S_1} = u}{\operatorname{Res}} \,\underset{\hat{w}_{S_1-1} = q {w}_{S_1}}{\operatorname{Res}} \cdots \underset{\hat{w}_2 = q {w}_3}{\operatorname{Res}} \, \underset{\hat{w}_1 = q {w}_2}{\operatorname{Res}}f$ into the same pattern as in the original form (\ref{recursiveFormOfF}). One could check that setting $S_1 = 1$ in equation (\ref{firstmethod}) is valid. 
	
	If one takes the following evaluation sequence of simple poles by induction
	\begin{align}
	\hat{w}_1 = u_1, \hat{w}_2 = q u_1, \cdots, \hat{w}_{S_1-1} = q^{S_1-2}u_1,\hat{w}_{S_1} = q^{S_1-1} u_1
	\end{align}
	we get
	\begin{align}
	&\underset{\hat{w}_{S_1} = q^{S_1-1}u_1}{\operatorname{Res}} \cdots \underset{\hat{w}_{2} = qu_1}{\operatorname{Res}}\underset{\hat{w}_{1} = u_1}{\operatorname{Res}}f 
	= \prod_{S_1<i}^{d}\left( \frac{{w}_i - q^{S_1-1} u_1}{{w}_i-  q^{S_1} u_1}\prod_{u \in U\backslash\{u_1\}}\frac{{w}_i - q^{-1}u}{{w}_i - u}\prod_{i<j}\frac{({w}_i-{w}_j)^2}{({w}_i-q{w}_j)(q{w}_i - {w}_j)}\right) \nonumber
	\\ & \cdot u_{1}^{S_1}\prod_{ k =  0}^{S_1 - 1} \prod_{u \in U \backslash \{u_1\} } \frac{q^k u_1-q^{-1}u}{q^k u_1-u}  \cdot  (q-1)^{S_1} q^{S_1 (S_1 -1 - d)} g(u_1,qu_1,\cdots,q^{S_1 - 1}{w}_{S_1}, {w}_{S_1 + 1},\cdots {w}_d)  \label{u1}
	\end{align}
	which agrees with the equation (\ref{firstmethod}), since $g$ is a symmetric function. That is to say, we get the same results from two different evaluation sequences
	\begin{align} \label{onebyoneVSall}
	\underset{\hat{w}_{S_1} = q^{S_1-1}u_1}{\operatorname{Res}} \cdots \underset{\hat{w}_{2} = qu_1}{\operatorname{Res}}\underset{\hat{w}_{1} = u_1}{\operatorname{Res}}f   = \underset{\hat{w}_{S_1} = u_1}{\operatorname{Res}} \,\underset{\hat{w}_{S_1-1} = q {w}_{S_1}}{\operatorname{Res}} \cdots \underset{\hat{w}_2 = q {w}_3}{\operatorname{Res}} \, \underset{\hat{w}_1 = q {w}_2}{\operatorname{Res}}f  
	\end{align} 
			
	As the evaluation process for sequence (\ref{evaluateProcess1}), we now picking up residues of $\tilde{f}$ in the following sequence
	\begin{align}
	\hat{w}_{S_{1}+1} = q {w}_{S_1 +2} \quad \hat{w}_{S_{1}+2} = q {w}_{S_1 +3}\quad \cdots \quad \hat{w}_{S_{1}+S_{2}-1} = q {w}_{S_1 +S_2}
	\end{align}
	Suppose $\hat{w}_{S_1+S_2}=u_2$, we have two cases here, i.e. 
	$u_2 \ne q^{S_1} u_1$ or $u_2 = q^{S_1} u_1$. By a little bit of computation, we obtain \\
	
	{\bf{Case 1}}: $u_2 \ne q^{S_1} u_1$,
	\begin{align}\label{udifferent}
	&\underset{\hat{w}_{S_1+S_2} = u_2}{\operatorname{Res}} \, \underset{\hat{w}_{S_1+S_2-1} = q {w}_{S_1+S_2}}{\operatorname{Res}} \cdots \underset{\hat{w}_{S_{1}+2} = q {w}_{S_{1}+3}}{\operatorname{Res}}\,\underset{\hat{w}_{S_{1}+1} = q {w}_{S_{1}+2}}{\operatorname{Res}} \, \underset{\hat{w}_{S_1} = u_1}{\operatorname{Res}} \,\underset{\hat{w}_{S_1-1} = q {w}_{S_1}}{\operatorname{Res}} \cdots \underset{\hat{w}_2 = q {w}_3}{\operatorname{Res}} \, \underset{\hat{w}_1 = q {w}_2}{\operatorname{Res}}f \nonumber
	\\ =&\underset{\hat{w}_{S_1+S_2} = u_1}{\operatorname{Res}} \, \underset{\hat{w}_{S_1+S_2-1} = q {w}_{S_1+S_2}}{\operatorname{Res}} \cdots \underset{\hat{w}_{S_{2}+2} = q {w}_{S_{2}+3}}{\operatorname{Res}}\,\underset{\hat{w}_{S_{2}+1} = q {w}_{S_{2}+2}}{\operatorname{Res}} \, \underset{\hat{w}_{S_2} = u_2}{\operatorname{Res}} \,\underset{\hat{w}_{S_2-1} = q {w}_{S_2}}{\operatorname{Res}} \cdots \underset{\hat{w}_2 = q {w}_3}{\operatorname{Res}} \, \underset{\hat{w}_1 = q {w}_2}{\operatorname{Res}}f 
	\end{align}

	{\bf{Case 2}}: $u_2 = q^{S_1} u_1$, 
	\begin{align}\label{usame}
	&\underset{\hat{w}_{S_1+S_2} = q^{S_1} u_1}{\operatorname{Res}} \, \underset{\hat{w}_{S_1+S_2-1} = q {w}_{S_1+S_2}}{\operatorname{Res}} \cdots \underset{\hat{w}_{S_{1}+2} = q {w}_{S_{1}+3}}{\operatorname{Res}}\,\underset{\hat{w}_{S_{1}+1} = q {w}_{S_{1}+2}}{\operatorname{Res}} \, \underset{\hat{w}_{S_1} = u_1}{\operatorname{Res}} \,\underset{\hat{w}_{S_1-1} = q {w}_{S_1}}{\operatorname{Res}} \cdots \underset{\hat{w}_2 = q {w}_3}{\operatorname{Res}} \, \underset{\hat{w}_1 = q {w}_2}{\operatorname{Res}}f \nonumber
	\\ =& \underset{\hat{w}_{S_1+S_2} = q^{S_1 + S_2 -1} u_1}{\operatorname{Res}} \, \underset{\hat{w}_{S_1+S_2-1} = q^{S_1 + S_2 -2} u_1}{\operatorname{Res}} \cdots \underset{\hat{w}_{S_{1}+2} = q^{S_1 + 1} u_1}{\operatorname{Res}}\,\underset{\hat{w}_{S_{1}+1} = q^{S_1} u_1}{\operatorname{Res}} \, \underset{\hat{w}_{S_1} = q^{S_1 -1}u}{\operatorname{Res}} \cdots \underset{\hat{w}_2 = q u_1}{\operatorname{Res}} \, \underset{\hat{w}_1 = u_1}{\operatorname{Res}}f
	\end{align}
	
	To summarize all above, we can repeat using above arguments to integrating all variables for the integrand of the form as in (\ref{recursiveFormOfF}) with one variable less each time. 
	
	When there is only one variable left
\begin{align}
	f(w) = g(w)\prod_{u \in U} \frac{w-q^{-1}u}{w - u}
\end{align}
we still update the set $U$ to $U\backslash\{u\} \cup \{qu\}$ after choosing pole at $\hat{w} = u \in U$. Using same argument to get (\ref{udifferent}) and (\ref{usame}), after picking up poles for all $w_i,  i \in [d]$, the results only depends on the final set $U$, and final set $U$ must be of the form
\begin{align}
	\{q^{d_1} x_1, \cdots , q^{d_r}x_r\}
\end{align}
where $d_1 + \cdots + d_r = d$, which means for each sequence, the final result can be indexed by a $r$-tuple partition of $d$.

Suppose there is a sequence with final set $\{q^{d_1} x_1, \cdots , q^{d_r}x_r\}$, then we can compute the result by following sequence
		\begin{align}
		(\hat{w}_1,\cdots,\hat{w}_d) = (x_1,qx_1\cdots,q^{d_1-1}x_1,x_2\cdots,x_r,\cdots qx_r,q^{d_r-1}x_r)
		\end{align}
		and note that we can actually do permutation on the left side, so for each partition $|\vec{d}| = d$, we have $d!$ possible evaluation sequences.
	
In all we get following lemma to compute $E_d$.
\begin{Lemma}{\label{iterated-residue}}
	We can write $E$ as 
\begin{align}
	E_d =\sum_{|\vec{d}|=d}d!E_{\vec{d}}		 \label{summand-d}	
\end{align}
where
\begin{align}
	E_{\vec{d}}= \lim_{w_{d}\rightarrow \hat{w}_{d} }\cdots\lim_{w_{1}\rightarrow \hat{w}_{1} }\left(\prod_{i=1}^{n}(w_i-\hat{w}_{i})f(\vec{w})\right )
\end{align}
here 
\begin{align}
	(\hat{w}_{1},\dots,\hat{w}_{d})=(x_1,qx_1,\dots,q^{d_1-1}x_1,x_2,q x_2,\dots,q^{d_2-1}x_2, \dots,x_r,\dots,q^{d_r -1}x_r) \nonumber
\end{align}
and the order to take limit is from $w_1$ to $w_d$.
\end{Lemma}

We now evaluate one specific configuration of these simple pole residues for given $\vec d$ by changing of variables:
\begin{align}
w^{i}_{n_i}=x_i q^{n_i - 1}z_{n_i}^{i} \quad\quad i=1,\dots,r\quad n_i=1,\dots,d_i \nonumber
\end{align}
{\bf{Notations:}} From now on, we would frequently use the following notations
\begin{align}
x_{ij} := x_i / x_j \ \ \ \ \ n_{ij} := n_i-n_j	
\end{align}

so
\begin{align}
f(\vec{w})=& \frac{1}{(1-q)^{d} d!\prod_{i,n_i}z^i_{n_i}}\cdot \displaystyle \prod_{i=1}^{r}\prod_{n_i\neq n_j}^{d_i}           \frac{1-q^{n_{ij}} z_{n_i}^{i}/z_{n_j}^{i}}{1-q^{n_{ij}+1} z_{n_i}^{i}/z_{n_j}^{i}} \nonumber \\
& \times \prod_{i,j=1|i\neq j}^{r}\prod_{n_i=1}^{d_i}\prod_{n_j=1}^{d_j}\frac{1-q^{n_{ij}} z_{n_i}^{i}/z_{n_j}^{j}x_{ij}}{1-q^{n_{ij}+1} z_{n_i}^{i}/z_{n_j}^{j}x_{ij}} \nonumber \\ 
&\times  \frac{\prod_{i}^{r}\prod_{n_i=1}^{d_i}(x_i q^{n_i - 1}z_{n_i}^{i})^{l-1}}{ \prod_{i,j=1|i\neq j}^{r}\prod_{n_i=1}^{d_i} (1-x_{ji}q^{1-n_i}/z^{i}_{n_i})} \nonumber \\
&\cdot \frac{1}{ \prod_{i=1}^{r}\prod_{n_i=1}^{d_i} (1-q^{1-n_i}/z^{i}_{n_i})} \cdot \frac{1}{ \prod_{i=1}^{r}\prod_{j=r+1}^{n}\prod_{n_i=1}^{d_i} (1-x_{ij}q^{n_i}z^{i}_{n_i})} \nonumber
\end{align}
now obtain grid of the simple pole terms and evaluate the function with $z_i=1$, note that
\begin{align*}
	\lim_{z^i_{d_i}\rightarrow 1 }\cdots\lim_{z^i_{1}\rightarrow 1 }\left(\prod_{n_i=1}^{d_i}(z_{n_i}^{i}  - 1 ) \cdot \frac{1}{\left(1-(z^i_{1})^{-1}\right)\left(1-z^i_{1}/z^i_2\right)\ldots \left(1-z^i_{d_i-1}/z^i_{d_i}\right) z^i_1 \cdots z^i_{d_i}}\right ) =1 \nonumber
\end{align*}
where the order to take limits is from $z_1^i$ to $z_{d_i}^i$. So this specific configuration of residues is 
\begin{align}
E_{\vec{d}}=&\frac{1}{(1-q)^{d} d!}\cdot \displaystyle \prod_{i=1}^{r}\left(\prod_{n_i\neq n_j|n_{ij}\neq -1}^{d_i} \frac{1-q^{n_{ij}} }{1-q^{n_{ij}+1} } \cdot \prod_{n_i=2}^{d_i}\frac{1-q^{-1}}{1-q^{1-n_i}} \right)\cdot \nonumber \\
&\times \prod_{i,j=1|i\neq j}^{r}\prod_{n_i=1}^{d_i}\prod_{n_j=1}^{d_j}\frac{1-q^{n_{ij}} x_{ij}}{1-q^{n_{ij}+1} x_{ij}} \nonumber \\ 
&\times  \frac{\prod_{i}^{r}\prod_{n_i=1}^{d_i}(x_i q^{n_i - 1})^{l}}{ \prod_{i,j=1|i\neq j}^{r}\prod_{n_i=1}^{d_i} (1-x_{ji}q^{1-n_i})}
\cdot \frac{1}{ \prod_{i=1}^{r}\prod_{j=r+1}^{n}\prod_{n_i=1}^{d_i} (1-x_{ij}q^{n_i})} \label{E_d}
\end{align}
and factors in $E_{\vec{d}}$ only with $x_{ij}$ with $i=1,\ldots,r$ and $j=r+1,\ldots,n$ are 
\begin{align}
E^{1}_{\vec{d}}:=\frac{1}{ \prod_{i=1}^{r}\prod_{j=r+1}^{n}\prod_{n_i=1}^{d_i} (1-x_{ij}q^{n_i})}=\frac{1}{\prod_{i=1}^{r}\prod_{j=r+1}^{n}(qx_{ij};q)_{d_i}} \nonumber
\end{align}
the factors in $E_{\vec{d}}$ not containing any $x_{ij}$ are
\begin{align}
E_{\vec{d}}^2:=\frac{1}{(1-q)^{d}}\cdot \displaystyle \prod_{i=1}^{r}\left(\prod_{n_i\neq n_j|n_{ij}\neq -1}^{d_i}           \frac{1-q^{n_{ij}} }{1-q^{n_{ij}+1} } \cdot \prod_{n_i=2}^{d_i}\frac{1-q^{-1}}{1-q^{1-n_i}} \right) \nonumber
\end{align}

Define $P_d$ as
\begin{align*}
P_d: = \left\{ 
\begin{array}{rcl}
& \displaystyle \prod_{i\neq j|i-j\neq -1}^{d} \frac{1-q^{i-j} }{1-q^{i-j+1} } \cdot \prod_{i=2}^{d_i}\frac{1-q^{-1}}{1-q^{1-i}}  \quad\quad   & d>1 \\
&	1                               & d=0,1
\end{array}
\right.
\end{align*}
by simple induction, we obtain
\begin{align}
\frac{P_d}{(1-q)^d}=\frac{1}{(q;q)_d} \quad\quad\quad d\geq 0 \nonumber
\end{align}
and 
\begin{align}
E_{\vec{d}}^2=\prod_{i=1}^{r}\frac{P_{d_i}}{(1-q)^{d_i}}=\prod_{i=1}^{r}\frac{1}{(q;q)_{d_i} }\nonumber
\end{align}
Factors left in (\ref{E_d}) are 
\begin{align}
E_{\vec{d}}^3:=& \left( \prod_{i,j=1|i\neq j}^{r}\prod_{n_i=1}^{d_i}\prod_{n_j=1}^{d_j}\frac{1-q^{n_{ij}} x_{ij}}{1-q^{n_{ij}+1} x_{ij}} \right) \frac{\prod_{i=1}^{r}\prod_{n_i=1}^{d_i}(x_i q^{n_i - 1})^l}{ \prod_{i,j=1|i\neq j}^{r}\prod_{n_i=1}^{d_i} (1-x_{ji}q^{1-n_i})}  \nonumber \\
=&\prod_{i\neq j}^{r}\prod_{n_j=1}^{d_j}\left( \left(\prod_{n_i=1}^{d_i} \frac{1-q^{n_{ij}} x_{ij}}{1-q^{n_{ij}+1} x_{ij}} \right) \cdot \frac{1}{1-x_{ij}q^{1-n_j}}\right)  \cdot \prod_{i=1}^{r}x_{i}^{ld_i}q^{\frac{ld_i(d_i-1)}{2}} \nonumber \\
=&\prod_{i\neq j}^{r}\prod_{n_j=1}^{d_j}\frac{1}{1-q^{d_i-n_j+1}x_{ij}} \cdot \prod_{i=1}^{r}x_{i}^{ld_i}q^{\frac{ld_i(d_i-1)}{2}}  \nonumber \\
=&\prod_{i=1}^{r}x_{i}^{ld_i}q^{\frac{ld_i(d_i-1)}{2}}\cdot \prod_{i\neq j}^{r}\frac{1}{(q^{d_{ij}+1};q)_{d_j}} \nonumber
\end{align}
where we swap indices $i$ and $j$ in factor $(1-x_{ji}q^{1-n_i})$. Note that $(q^{d_{ij}+1}x_{ij},q)_{d_j}=(q;q)_{d_i}$ when $i=j$. In all, we have
\begin{align*}
d!E_{\vec{d}}&=d!E_{\vec{d}}^1E_{\vec{d}}^2E_{\vec{d}}^3	 \\
             &=\frac{\prod_{i=1}^rx_{i}^{ld_i}q^{\frac{ld_i(d_i-1)}{2}}}{\prod_{i,j=1}^r(q^{d_{ij}+1}x_{ij};q)_{d_j}\prod_{i=r}^r\prod_{j=r+1}^{n}(qx_{ij};q)_{d_i}}
\end{align*}
comparing $d!E_{\vec{d}}$ with definition of $A_{d}(\vec{x},I,l)$ in (\ref{A_d}), we have

The above equations prove that the summand in (\ref{summand-d}) corresponding to given $\vec d$ equals to one summand in $A_{d}\left(\vec{x},I,l\right)$, thus, we have
\begin{align}
A_{d}\left(\vec{x},I,l\right) = \sum_{|\vec{d}|=d}d!E_{\vec{d}}	=E_d \nonumber
\end{align}
where $I=\{1,2,\cdots,r \}$.

On the other hand, we consider the integration
\begin{align}
F_d :=\int_{C^{\prime}_{\rho}} \frac{dw_{i_d}}{2\pi\sqrt{-1}}  \ldots \int_{C^{\prime}_{\rho}}  \frac{dw_{i_1}}{2\pi\sqrt{-1}} f(w_{i_1},\cdots,w_{i_d}) \label{integration-reverse}	
\end{align}
where the integration contour for each variable $w_i$ is still $|w_i|=\rho $, but, with clockwise direction. By definition, we can compute this integration by taking iterated residues outside $w=\rho $. The difference of degree of denominator and that of numerator for each variable $w_i$ is $n-l+1-r \geq 2 $ which is guaranteed by condition (\ref{level-condition}), so residue for each variable at infinity is 0.

The iterated residues in this case are similar to the previous counter-clockwise direction. Similar arguments as (\ref{iterated-residue}) shows
\begin{align}
	F_d=\sum_{|\vec{d^\prime}|=d}d!F_{\vec{d^\prime}}	 \nonumber
\end{align}
where
\begin{align}
    F_{\vec{d^\prime}}= \lim_{w_{d}\rightarrow \hat{w}_{d} }\cdots\lim_{w_{1}\rightarrow \hat{w}_{1} }\left(\prod_{i=1}^{n}(w_i-\hat{w}_{i})f(\vec{w})\right ) \nonumber
\end{align}
here
\begin{align}
\left\{\hat{w}_{1}, \ldots, \hat{w}_{d}\right\}=\left\{x_{r+1}q^{-1}, x_{r+1}q^{-2}, \ldots, x_{r+1} q^{-d_{r+1}}, \ldots, x_{n}q^{-1},  x_{n} q^{-2}, \ldots, x_{n} q^{-d_{n}} \right\} \nonumber
\end{align} 
and the order to take limit is from $w_1$ to $w_d$.

For a given partition $\vec{d}^{\prime}=(d_{r+1}, \cdots, d_n )$, we compute $E_{\vec{d}}^{\prime}$ by doing the following change of variables
\begin{align}
	w_{n_{i}}^{i}=x_{i} q^{-n_{i}} z_{n_{i}}^{i} \quad i=r+1, \ldots, n \quad n_{i}=1, \ldots, d_{i} \nonumber
\end{align}

\begin{align}
f(\vec{w})=& \frac{1}{(1-q)^{d} d!\prod_{i,n_i}z^i_{n_i}}\cdot  \prod_{i=r+1}^{n}\prod_{n_i\neq n_j}^{d_i}          \frac{1-q^{n_{ji}} z_{n_i}^{i}/z_{n_j}^{i}}{1-q^{n_{ji}+1} z_{n_i}^{i}/z_{n_j}^{i}} \nonumber \\
& \times \prod_{i,j=r+1|i\neq j}^{n}\prod_{n_i=1}^{d_i}\prod_{n_j=1}^{d_j}  \frac{1-q^{n_{ji}} z_{n_i}^{i}/z_{n_j}^{j}x_{ij}}{1-q^{n_{ji}+1} z_{n_i}^{i}/z_{n_j}^{j}x_{ij}} \nonumber \\ 
&\times  \frac{\prod_{i=r+1}^{n}\prod_{n_i=1}^{d_i}(x_i q^{-n_i}z_{n_i}^{i})^{l-1}}{ \prod_{i,j=r+1|i\neq j}^{n}\prod_{n_i=1}^{d_i} (1-x_{ij}q^{1-n_i}z^{i}_{n_i})} \nonumber \\
&\cdot \frac{1}{ \prod_{i=r+1}^{n}\prod_{n_i=1}^{d_i} (1-q^{1-n_i}z^{i}_{n_i})} \cdot \frac{1}{ \prod_{i=r+1}^{n}\prod_{j=1}^{r}\prod_{n_i=1}^{d_i} (1-x_{ji}q^{n_i}/z^{i}_{n_i})} \nonumber
\end{align}
note that 

\begin{align}
 	\lim_{z^i_{d_i}\rightarrow 1 }\cdots\lim_{z^i_{1}\rightarrow 1 }\left(\prod_{n_i=1}^{d_i}(z_{n_i}^{i}  - 1 ) \cdot \frac{1}{\left(1-z^i_{1}\right)\left(1-z^i_{2}/z^i_1\right)\ldots \left(1-z^i_{d_i}/z^i_{d_i-1}\right)z^i_1 \cdots z^i_{d_i} }\right) =(-1)^{d_i}  \nonumber
\end{align}
where the order to take limits is from $z_1^i$ to $z_{d_i}^i$. So the residues for one specific configuration of residues of type $\vec{d'}$ is 
\begin{align}
F_{\vec{d}^{\prime}}=&\frac{(-1)^d}{(1-q)^{d} d!}\cdot \displaystyle \prod_{i=r+1}^{n}\left(\prod_{n_i\neq n_j|n_{ji}\neq -1}^{d_i}           \frac{1-q^{n_{ji}} }{1-q^{n_{ji}+1} } \cdot \prod_{n_i=2}^{d_i}\frac{1-q^{-1}}{1-q^{1-n_i}} \right) \nonumber \\
&\times \prod_{i,j=r+1|i\neq j}^{n}\prod_{n_i=1}^{d_i}\prod_{n_j=1}^{d_j}\frac{1-q^{n_{ji}} x_{ij}}{1-q^{n_{ji}+1} x_{ij}} \nonumber \\ 
&\times  \frac{\prod_{i=r+1}^{n}\prod_{n_i=}^{d_i}(x_i q^{-n_i})^{l}}{\prod_{i,j=r+1|i\neq j}^{n}\prod_{n_i=1}^{d_i} (1-x_{ij}q^{1-n_i})}
\cdot \frac{1}{ \prod_{i=r+1}^{n}\prod_{j=1}^{r}\prod_{n_i=1}^{d_i} (1-x_{ji}q^{n_i})} \nonumber
\end{align}
after almost same computation as for $E_{\vec{d}}$, we can simplify the above equation to 
\begin{align}
(-1)^d\frac{\prod_{i=r+1}^{n}x_{i}^{ld_i}q^{-\frac{ld_i(d_i+1)}{2}}}{\prod_{i=r+1}^{n}(q; q)_{d_i} \prod_{i \neq j|i,j=r+1 }^{n}\left(q^{d_{ij}+1}x_{ji};q \right)_{d j}\prod_{i=r+1}^{n} \prod_{j=1}^{r} (qx_{ji} ;q)_{d_i}} \nonumber
\end{align}
which proves
\begin{align}
F_d=(-1)^d B\left( \vec{x},I^{\complement } ,-l \right) \nonumber
\end{align}
Since the residue at infinity is zero, using Cauchy Residue Theorem $d$ times,
\begin{align*}
\int_{C_{\rho}} \ldots \int_{C_{\rho}} f(\vec w) \frac{dw_1}{2\pi\sqrt{-1}w_1} \ldots \frac{dw_d}{2\pi\sqrt{-1}w_d}\\
=(-1)^d\int_{C^{'}_{\rho}} \ldots \int_{C^{'}_{\rho}} f(\vec w) \frac{dw_1}{2\pi\sqrt{-1}w_1} \ldots \frac{dw_d}{2\pi\sqrt{-1}w_d}
\end{align*}
then we arrive at (\ref{Thm-A}) (\ref{Thm-B}) and (\ref{Thm-1}) of the following proposition stated in the introduction:
\begin{Proposition}	
	Denoted by $[n]$ the set of elements $\{ 1, \ldots,n \}$, let $ \emptyset \neq I \subsetneq [n] $ be a subset of $[n]$, $|I|$ be its cardinality, and denoted by $I^\complement$ the complementary set of $I$ in $[n]$. Then for constant positive integers  $d$, $n$ and integer $l$ with restriction: $1-|I|\leq l\leq n-|I|-1$  , let $A_d\left(   \vec{x} , I ,l\right)$ and $B_d\left(  \vec{x}, I,l\right)$ be two rational functions in $\vec{x}$ and $q$ with an extra data $l$.
	\begin{align}
	A_d\left(  \vec{x}, I, l\right)=&\sum_{| {d}_I|=d}\frac{\left( \prod_{i \in I}x_{i}^{d_i}q^{\frac{d_i(d_i-1)}{2}}\right)^l}{\prod_{i ,j \in I}\left( q^{d_{ij}+1}x_{ij};q \right)_{d j}\prod_{i \in I } \prod_{j \in I^\complement } (qx_{ij};q )_{d_i}}   \label{Thm-A} \\
	B_d\left(  \vec{x}, I, l\right)=&\sum_{|\vec{d}_{I}|=d}\frac{\left(\prod_{i \in I }x_{i}^{-d_i}q^{\frac{d_i(d_i+1)}{2}}\right)^l}{ \prod_{i , j \in I }\left( q^{d_{ij}+1}x_{ji};q \right)_{d j}\prod_{i \in I } \prod_{j \in I^\complement } (qx_{ji} ;q)_{d_i}} \label{Thm-B}
	\end{align}
	where $\vec{d}_I$ is $|I|$-tuple of non negative integers, and $|\vec{d}_I|:=\sum_{i \in I} d_i$. $x_i, i=1,...,n$ are parameters. For convenience,we use the notation $x_{ij}:=x_{i}/x_{j}$ and $d_{ij}:=d_i-d_j$. Then we have 
	\begin{align}
		A_d\left(  \vec{x},I,l\right)= 	B_d\left(  \vec{x},I^\complement,-l\right)  {\label{Thm-1}}
    \end{align}
\end{Proposition}

\subsection{Examples}
\label{sec:3}
In the following two examples, we show how the proof of Proposition \ref{Thm-0} works.
\begin{Example}[d=1\label{d=1}]
 For the case l=0, d=1, r=2, n=3	. (\ref{f}) becomes the following simple form
 \begin{align*}
 f(w)=&\frac{1}{(1-q) }  \frac{w^{-1}}{(1-x_1/w)(1-x_2/w) (1-qw/x_{3})} 	
 \end{align*}
 Consider integration (\ref{integration}), then there are simple poles of type $(1,0)$ and $(0,1)$ in the counter $C_{\rho}$:
 \begin{itemize}
 \item type (1,0): $ w  =  x_1 $  
 \item type (0,1): $ w  =  x_2 $
 \end{itemize}
Then the residue for each type comes as follows:
 \begin{itemize}
 \item type ${(1,0)}$:  
   \begin{align*}
   E_{(1,0)}=\underset{\hat{w} = x_1}{\operatorname{Res}}f = \frac{1}{(1-q)(1-x_{21})(1-qx_{13})} 
   \end{align*}
 \item type ${(0,1)}$: 
   \begin{align*}
    E_{(0,1)}=\underset{\hat{w} = x_2}{\operatorname{Res}}f= \frac{1}{(1-q)(1-x_{12})(1-qx_{23})}
   \end{align*}
 \end{itemize}
 and there is only one simple pole $w=q^{-1}x_3 $ in the counter $C^{\prime}_{\rho}$, so 
 \begin{itemize}
 \item type  $1$:	
 \begin{align*}
 	E^{\prime}_{1}= \underset{\hat{w} = q^{-1}x_3}{\operatorname{Res}}f = \frac{-1}{(1-q)(1-qx_{13})(1-qx_{23})}
 \end{align*}
 \end{itemize}
 and it is easy to obtain 
 \begin{align*}
 	\frac{1}{(1-q)(1-x_{21})(1-qx_{13})} +\frac{1}{(1-q)(1-x_{12})(1-qx_{23})} = \frac{1}{(1-q)(1-qx_{13})(1-qx_{23})}
 \end{align*}
 which agrees with (\ref{Thm-1}).  
 \end{Example}

\begin{Example}[d=2\label{d=2}]
 For the case l=0, d=2, r=2, n=3	. (\ref{f}) becomes the following simple form
 \begin{align*}
 f(\vec w)=&\frac{1}{2(1-q)^2 }\prod_{i \neq j}^2 \frac{1-w_{i}/w_{j}}{1-qw_{i}/w_{j}}\prod_{i=1}^2 \frac{w_i^{-1}}{\prod_{j=1}^{2}(1-x_j/w_{i}) \cdot (1-qw_{i}/x_{3})} 	
 \end{align*}
 Consider integration (\ref{integration}), then there are simple poles of type $(2,0)$, $(1,1)$ and $(0,2)$ in the counter $C_{\rho_i}$:
 \begin{itemize}
 \item type (2,0): $\{ w_1, w_2 \} = \{ x_1, x_1 q \} $  
 \item type (1,1): $\{ w_1, w_2 \} = \{ x_1, x_2  \} $  
 \item type (0,2):  $\{ w_1, w_2 \} = \{ x_2, x_2 q \} $  
 \end{itemize}
Then the residue for each type comes as follows:
 \begin{itemize}
 \item $type \ {(2,0)}$:  
     \begin{align*}
   2!E_{(2,0)} &=	\underset{\hat{w}_{2}= qx_1}{\operatorname{Res}}\underset{\hat{w}_{1} = x_1}{\operatorname{Res}}f+\underset{\hat{w}_{2}= x_1}{\operatorname{Res}}\underset{\hat{w}_{1} = qw_2}{\operatorname{Res}}f \\ 
   &=\frac{1}{2(1-q)^2} \frac{1}{(1+q)(1-q x_{13})(1-q^2x_{13})(1-x_{21})(1-q^{-1}x_{21})} \\
   	&+   \frac{1}{2(1-q)^2} \frac{1}{(1+q)(1-q^2x_{13})(1-qx_{13})(1-q^{-1}x_{21})(1-x_{21})}  \\ 
   &= \frac{1}{(1-q)^2} \frac{1}{(1+q)(1-q^2 x_{13})(1-q x_{13})(1-q^{-1} x_{21})(1- x_{21})}
     \end{align*}
 \item type  ${(1,1)}$: 
 \begin{align*}
    2!E_{(1,1)} &= \underset{\hat{w}_{2} = x_2}{\operatorname{Res}}\underset{\hat{w}_{1} = x_1}{\operatorname{Res}}f +\underset{\hat{w}_{2} = x_1}{\operatorname{Res}}\underset{\hat{w}_{1} = x_2}{\operatorname{Res}}f \\ 
    &= \frac{1}{2(1-q)^2} \frac{1}{(1-qx_{12})(1-qx_{21})(1-qx_{13})(1-qx_{23})} \\
    &+  \frac{1}{2(1-q)^2} \frac{1}{(1-qx_{21})(1-qx_{12})(1-qx_{23})(1-qx_{13})} \\
    &= \frac{1}{(1-q)^2} \frac{1}{(1-q x_{21})(1-q x_{12})(1-q x_{23})(1-q x_{13})}
 \end{align*}
 \item type ${(0,2)}$: 
 \begin{align*}
2!E_{(0,2)} &= \underset{\hat{w}_{2} = qx_2}{\operatorname{Res}}\underset{\hat{w}_{1} = x_2}{\operatorname{Res}}f+\underset{\hat{w}_{2} = x_2}{\operatorname{Res}}\underset{\hat{w}_{1} = qw_2}{\operatorname{Res}}f \\
&= \frac{1}{2(1-q)^2}	\frac{1}{(1+q)(1-x_{12})(1-qx_{23})(1-q^{-1}x_{12})(1-q^2x_{23})} \\
&+ \frac{1}{2(1-q)^2}\frac{1}{(1+q)(1-q^{-1}x_{12})(1-q^2x_{23})(1-x_{12})(1-qx_{23})} \\
&= \frac{1}{(1-q)^2}\frac{1}{(1+q)(1-q^{-1} x_{12})(1-q^2 x_{23})(1- x_{12})(1-q x_{23})}
 \end{align*}	
 \end{itemize}
 Consider integration (\ref{integration-reverse}), then there are simple poles of type 2 in the counter $C^{\prime}_{\rho_i}$:
  \begin{itemize}
 \item type 2: $\{ w_1, w_2 \} = \{ q^{-1}x_3, q^{-2}x_3  \} $  
 \end{itemize}
Then the residue for each type 2 comes as follows:
\begin{itemize}
\item {type} $ {2}$:	
\begin{align*}
(-1)^22!E^{\prime}_{2} &= \underset{\hat{w}_{2} = q^{-2}x_3}{\operatorname{Res}}\underset{\hat{w}_{1} = q^{-1}x_3}{\operatorname{Res}}f+\underset{\hat{w}_{2} = q^{-1}x_3}{\operatorname{Res}}\underset{\hat{w}_{1} = q^{-1}w_2}{\operatorname{Res}}f \\
&=\frac{1}{(1+q)(1-q)^2(1-q^2 x_{13})(1-q x_{13})(1-q^2 x_{23})(1-q x_{23})} 
\end{align*}
\end{itemize}
by a little bit computation, we have
\begin{align*}
	E_2&=2!E_{(2,0)}+2!E_{(1,1)}+2!E_{(0,2)}  = E^{\prime}_2
\end{align*}
\end{Example}

\begin{Example}
From Proposition \ref{Thm-0}, if we take $n=3$, $l=0$ and $I=[2]$, we know that $A_d(\vec{x},[2],0)=B_d(\vec{x},[3]\backslash [2],0 )$. By the following computation, there is a phenomenon that we could extracting from $A_d(\vec{x},[2],0)$ to get $B_d(\vec{x},[3]\backslash [2],0 )$ times another factor, when $d=1,2$, i.e. $A_d(\vec{x},[2],0)=B_d(\vec{x},[3]\backslash [2],0 ) \times G(\vec{x}), \ d=1,2 $, thus we can conclude that $G(\vec{x})=1$. Furthermore, this is a general phenomenon for all $d$, see following Corollary \ref{identity=1}.

By definition $\vec{x}=\{x_1,x_2,x_3 \}$, so
\begin{align}
&A_d(\vec{x},[2],0)= \sum_{d_1+d_2=d}\frac{1}{(q;q)_{d_1}(q;q)_{d_2}(q^{d_{12}+1}x_{12};q)_{d_2}(q^{d_{21}+1}x_{12};q)_{d_1} (qx_{13};q)_{d_1}(qx_{23};q)_{d_2} }     \\
&B_d(\vec{x},[3]\backslash [2],0) = \frac{1}{(q;q)_d(qx_{13};q)_d(qx_{23};q)_d}	
\end{align}

For $d=1$,i.e. $(d_1,d_2)=(1,0)$ or $(0,1)$, we have
\begin{align}
& A_1(\vec{x},[2],0) =\sum_{d_1+d_2=1}\frac{1}{(q;q)_{d_1}(q;q)_{d_2}(q^{d_{12}+1}x_{12};q)_{d_2}(q^{d_{21}+1}x_{12};q)_{d_1} (qx_{13};q)_{d_1}(qx_{23};q)_{d_2} }  \nonumber  \\ 
=&\sum_{d_1+d_2=1}\frac{1}{(q;q)_1(qx_{13};q)_1(qx_{23};q)_1} \cdot\frac{(q;q)_1(qx_{13};q)_1(qx_{23};q)_1}{(q;q)_{d_1}(q;q)_{d_2}(q^{d_{12}+1}x_{12};q)_{d_2}(q^{d_{21}+1}x_{12};q)_{d_1} (qx_{13};q)_{d_1}(qx_{23};q)_{d_2} }      \nonumber \\
=&B_1(\vec{x},[3] \backslash [2] ,0) \times\sum_{(d_1,d_2)=(1,0),(0,1)}\frac{(q;q)_1}{(q;q)_{d_1}(q;q)_{d_2}}\prod_{i=1}^2 \left( \prod_{j \neq i}^{2} \frac{(q^{d_i+1}x_{i 3};q)_{1-d_i}}{(q^{d_{ij}+1}x_{i j};q)_{d_j} }\right) \nonumber \\
     =& B_1(\vec{x},[3] \backslash [2] ,0) \times \left(  \frac{1-qx_{23}}{1-x_{21}}+ \frac{1-qx_{13}}{1-x_{12}}   \right) \nonumber\\
     =& B_1(\vec{x},[3] \backslash [2] ,0)    \nonumber
\end{align}
for $d=2$,i.e.$(d_1,d_2)=(2,0),(1,1)$ or $(0,2)$, similarly, we have
\begin{align*}
&A_2(\vec{x},[2],0) = B_2(\vec{x},[3] \backslash [2] ,0)  \times  \sum_{(d_1,d_2)=(2,0),(1,1),(0,2)}\frac{(q;q)_2}{(q;q)_{d_1}(q;q)_{d_2}}\prod_{i=1}^2 \left( \prod_{j \neq i}^{2} \frac{(q^{d_i+1}x_{i 3};q)_{2-d_i}}{(q^{d_{ij}+1}x_{i j};q)_{d_j} }\right) \\
     &=  B_2(\vec{x},[3] \backslash [2] ,0)  \times \left( \frac{(1-qx_{13})(1-q^2x_{13})}{(1-q^{-1}x_{21})(1-x_{21})}+\frac{(1+q)(1-q^2x_{13})(1-q^2x_{23})}{(1-qx_{12})(1-qx_{21})} + \frac{(1-qx_{13})(1-q^2x_{13})}{(1-q^{-1}x_{12})(1-x_{12})}  \right) \\
     &=  B_2(\vec{x},[3] \backslash [2] ,0)  
\end{align*}	
\end{Example}

More generally, we have the following Corollary,
\begin{Corollary}\label{identity=1}
\begin{align*}
\sum_{d_1+d_2=d} \frac{(q;q)_d}{(q;q)_{d_1}(q;q)_{d_2}}   \prod_{j \neq i}^{2} \frac{(q^{d_i+1}x_{i 3};q)_{d-d_i}}{(q^{d_i-d_j+1}x_{i j};q)_{d_j}} =1
\end{align*}
\end{Corollary}
\begin{proof}
Set $l=0$, $r=2,n=3$ in (\ref{Thm-1}), we have
\begin{align}
	A_d(\vec{x},[2],0) =& \sum_{d_1+d_2= d}  \prod^2_{i,j=1}\frac{1}{(q^{d_{ij}+1}x_{ij};q)_{d_j}} \prod^2_{i=1}\frac{1}{(qx_{i3};q)_{d_i}} \nonumber \\
	=& \sum_{d_1+d_2= d}  \prod^2_{i=1}\left(  \frac{1}{(q;q)_{d_i}}\prod_{j \neq i}^2 \frac{1}{(q^{d_{ij}+1}x_{ij};q)_{d_j}} \cdot \frac{1}{(qx_{i3};q)_{d_i}}  \right)   \nonumber \\
    =& \sum_{d_1+d_2 = d}    \frac{(q^{d_1+1}x_{1 3};q)_{d-d_1}(q^{d_2+1}x_{2 3};q)_{d-d_2}}{(qx_{1 3};q)_{d}(qx_{2 3};q)_{d}}  \prod_{j \neq i}^2 \left( \frac{1}{(q;q)_{d_i}} \frac{1}{(q^{d_i-d_j+1}x_{i j};q)_{d_{j}}}\right) \nonumber \\
    =&  \sum_{d_1+d_2=d}  \frac{(q;q)_d}{(q;q)_d(qx_{1 3};q)_{d}(qx_{2 3};q)_{d}} \prod_{j \neq i}^2 \left(   \frac{(q^{d_i+1}x_{i 3};q)_{d-d_i}}{(q;q)_{d_i}(q^{d_{ij}+1}x_{i j};q)_{d_j} }\right) \nonumber \\
    =&  \sum_{d_1+d_2=d}B_d(\vec{x},[n] \backslash [2] ,0) \cdot \frac{(q;q)_d}{(q;q)_{d_1}(q;q)_{d_2}} \prod_{j \neq i}^2 \left(  \frac{(q^{d_i+1}x_{i 3};q)_{d-d_i}}{(q^{d_{ij}+1}x_{i j};q)_{d_j} }\right) \nonumber
    \end{align} 
since we know $A_d(\vec{x},I ,0)$ equals to $B_d(\vec{x},I^{\complement} ,0)$, we get the conclusion.
\end{proof}

\subsection{Boundary cases}
\label{sec:4}
For boundary cases $l=-r$, $l=n-r$, (\ref{Thm-1}) no longer holds, since the residue at infinity is nonzero. But we can add a new extra variable $x_{n+1}$ in (\ref{Thm-1}) and then taking limit to deduce some identites in boundary case, specifically
\begin{Corollary}
	\begin{itemize}
\item For $l=n-r$, we have
\begin{align}
	   A_d\left(  \vec{x},I,l\right)= \sum_{s=0}^d C_{s}(  \vec{x},I^\complement,d)	B_{d-s}\left(  \vec{x},q,I^\complement,-l\right) {\label{Thm-2}} 
\end{align}	
where $C_{s}( \vec{x},I,d) $ is defined as
\begin{align*}
C_{s}(  \vec{x},I,d)=\frac{(-1)^{|I|\cdot s}\prod_{i\in I^\complement}x_i^{s}}{(q; q)_{s} q^{s(d-s+|I|)}  }	
\end{align*}
\item For $l=-r$, we have
\begin{align}
	 B_d\left(  \vec{x},I^\complement,-l\right)= \sum_{s=0}^d D_{s}(  \vec{x},I,d)	A_{d-s}\left(  \vec{x},q,I,l\right) {\label{Thm-3}}
\end{align}
\begin{align*}
D_{s}(  \vec{x},I,d)=\frac{(-1)^{|I|\cdot s}\prod_{i \in I} x^{-s}_{i}}{ (q;q)_sq^{s(d-s)} }
\end{align*}
\end{itemize}
\end{Corollary}

\begin{proof}
Consider $[n+1]$, $I  \subsetneq [n+1] $, $\{n+1\} \notin I   $, $l=n-|I|$ in (\ref{Thm-1}), then we have 
\begin{align}
& \sum_{|\vec{d}_I|=d}\frac{\left( \prod_{i \in I}x_{i}^{d_i}q^{\frac{d_i(d_i-1)}{2}}\right)^l}{\prod_{i ,j \in I}\left( q^{d_{ij}+1}x_{ij};q \right)_{d j}\prod_{i \in I } \prod_{j \in I^\complement } (qx_{ij};q )_{d_i}} \label{LHS} \\
=& \sum_{|\vec{d}_{I^\complement}|=d}\frac{\left(\prod_{i \in I^\complement }x_{i}^{-d_i}q^{\frac{d_i(d_i+1)}{2}}\right)^{-l}}{ \prod_{i , j \in I^\complement }\left(q^{d_{ij}+1}x_{ji};q \right)_{d j}\prod_{i \in I^\complement } \prod_{j \in I } (qx_{ji} ;q)_{d_i}} \label{RHS}
\end{align}
It is easy to see taking  $\lim_{ x_{n+1} \rightarrow \infty }$ in (\ref{LHS}), we obtain 
\begin{align*}
	\lim_{ x_{n+1} \rightarrow \infty }(\ref{LHS})=A_d\left(  \vec{x},I,l \right), \ \ for \ \ l=n-|I|
\end{align*}
Now let's take limit $ \lim_{x_{n+1} \rightarrow \infty } $ in (\ref{RHS})
\begin{align}
&\sum_{|\vec{d}_{I^\complement}|=d}\frac{\left(\prod_{i \in I^\complement }x_{i}^{-d_i}q^{\frac{d_i(d_i+1)}{2}}\right)^{-l}}{ \prod_{i , j \in I^\complement }\left(q^{d_{ij}+1}x_{ji};q \right)_{d j}\prod_{i \in I^\complement } \prod_{j \in I } (qx_{ji};q )_{d_i}} \nonumber \\
=&\sum_{|\vec{d}_{I^\complement}|=d}  \frac{1}{(q;q)_{d_{n+1}}} \cdot \frac{1}{\prod_{j\in I}(qx_{j,n+1};q)_{d_{n+1}}} \cdot \frac{1}{\prod_{j \in \{ [n] \backslash I \} }(q^{d_{n+1}-d_j+1}x_{j,n+1};q)_{d_j}}  \label{two terms} \\
 \times &  \frac{(x^{d_{n+1}}_{n+1}q^{-\frac{d_{n+1}(d_{n+1}+1)}{2}})^l}{\prod_{i \in \{[n] \backslash I \} }(q^{d_i-d_{n+1}+1}x_{n+1,i};q)}  \label{limits}  \\
 \times & \prod_{i \in \{[n] \backslash I \} } \left(   \frac{1}{\prod_{j  \in \{[n] \backslash I \}   } (q^{d_i-d_j+1}x_{j i};q)_{d_{j}}} \cdot \frac{(x^{d_{i}}_{i}q^{\frac{-d_{i}(d_{i}+1)}{2}})^{l}}{ \prod_{j \in I }(qx_{j i};q)_{d_i}} \right)  \nonumber
\end{align}
the limits of last two terms in (\ref{two terms}) equal 1, and by a little bit computation, we obtain the limit of (\ref{limits}) equals
\begin{align*}
     (-1)^{d_{n+1}(n-r)}	\cdot q^{-(\sum_{i \in \{[n] \backslash I \} } d_i)d_{n+1}-(n-r)d_{n+1}} \prod_{i \in \{[n] \backslash I \}}x^{d_{n+1}}_i
\end{align*}
 then we obtain
\begin{align}
\lim_{x_{n+1} \rightarrow \infty} 	& \sum_{|\vec{d}_{I^\complement}|=d} \prod_{i \in I^\complement } \left(\frac{1}{   \prod_{j \in I^\complement}(q^{d_i-d_j+1}x_{j i};q)_{d_{j}}} \cdot \frac{(x^{d_{i}}_{i}q^{\frac{-d_{i}(d_{i}+1)}{2}})^{l}}{ \prod_{j \in I}(qx_{j i};q)_{d_i}} \right) \nonumber \\
=& \sum_{|\vec{d}_{I^\complement}|=d}  \frac{1}{(q;q)_{d_{n+1}}} \cdot \frac{(-1)^{d_{n+1}(n-|I|)}}{q^{(\sum_{i \in \{[n] \backslash I \}} d_i)d_{n+1}+(n-|I|)d_{n+1}}} \cdot \frac{1}{\prod_{i \in \{[n] \backslash I \}}x^{-d_{n+1}}_i}
\nonumber \\
 \times & \prod_{i \in \{[n] \backslash I \}} \left(   \frac{1}{\prod_{j \in \{[n] \backslash I \}}(q^{d_i-d_j+1}x_{j i};q)_{d_{j}}} \cdot \frac{(x^{d_{i}}_{i}q^{\frac{-d_{i}(d_{i}+1)}{2}})^{l}}{ \prod_{j \in I}(qx_{j i};q)_{d_i}} \right) \nonumber \\
=& \sum_{\alpha=0}^d  \frac{1}{(q;q)_{d-\alpha}} \cdot \frac{(-1)^{(d-\alpha)(n-|I|)}}{q^{(n-|I|+\alpha)(d-\alpha)}} \cdot \frac{1}{\prod_{i=r+1}^nx^{-(d-\alpha)}_i}  \cdot B_\alpha \left(  \vec{x},I^\complement,-l\right)  \nonumber
\end{align}
we obtain the conclusion. 

 Similarly, consider $A_d\left(  \vec{x}\cup x_{n+1} ,\tilde{I},l\right)$ and $B_d\left(  \vec{x} \cup x_{n+1},\tilde{I} ^\complement,-l\right)$, for $\tilde{I}= I \cup \{n+1\}$ and $l=-|I|$, from (\ref{Thm-1}) we have
\begin{align}
& \sum_{|\vec{d}_{\tilde{I}}| = d} \prod_{i \in \tilde{I} } \left(   \frac{1}{\prod_{j \in \tilde{I} }(q^{d_i-d_j+1}x_{i j};q)_{d_{j}}} \cdot \frac{(x^{d_{i}}_{i}q^{\frac{d_{i}(d_{i}-1)}{2}})^{l}}{ \prod_{j \in \tilde{I}^\complement }(qx_{i j};q)_{d_i}} \right) \label{LHS-(-r)} \\
=& \sum_{|\vec{d}_{\tilde{I}^\complement}| = d} \prod_{i \in \tilde{I}^\complement }\left(   \frac{1}{\prod_{j \in \tilde{I}^\complement }(q^{d_i-d_j+1}x_{j i};q)_{d_{j}}} \cdot \frac{(x^{d_{i}}_{i}q^{\frac{-d_{i}(d_{i}+1)}{2}})^{-l}}{ \prod_{j\in \tilde{I}}(qx_{j i};q)_{d_i}} \right) \label{RHS-(-r)}
\end{align}
It is easy to see that after taking $\operatornamewithlimits{lim}_{x_{n+1} \rightarrow 0} $ in (\ref{RHS-(-r)}), we obtain
\begin{align*}
B_d\left( \vec{x},I^\complement,l\right) , \ \ for \ \ l=-|I|
\end{align*}
First, rewrite (\ref{LHS-(-r)}) as follows,
\begin{align}
& \sum_{|\vec{d}_{\tilde{I}}| = d} \prod_{i \in \tilde{I} }  \left(  \frac{1}{ \prod_{j \in \tilde{I} }(q^{d_i-d_j+1}x_{i j};q)_{d_{j}}} \cdot \frac{(x^{d_{i}}_{i}q^{\frac{d_{i}(d_{i}-1)}{2}})^{-|I|}}{ \prod_{j\in \tilde{I}^\complement }(qx_{i j};q)_{d_i}} \right) \nonumber  \\
=&\sum_{|\vec{d}_{\tilde{I}}| = d} \frac{(\prod_{i \in {I} }x_i^{d_i}q^{\frac{d_i(d_i-1)}{2}})^{-|I|}}{\prod_{i , j \in {I} }(q^{d_{ij}+1} x_{ij};q)_{d_j} \prod_{i \in {I} }\prod_{j\in \{ [n] \backslash I\} }(qx_{ij};q)_{d_i}} \nonumber \\
\times & \frac{(x^{d_{n+1}}_{n+1}q^{\frac{d_{n+1}(d_{n+1}-1)}{2}})^{-|I|}}{(q;q)_{d_{n+1}}   \prod_{i \in I }(q^{d_i-d_{n+1}+1}x_{i,n+1};q)_{d_{n+1}} \prod_{j \in I}(q^{d_{n+1}-d_{j}+1}x_{n+1,j};q)_{d_{j}} \prod_{j \in \{ [n] \backslash I \} }(q x_{n+1,j};q)_{d_{n+1}}} \nonumber
\end{align}
Now let's take limit $ \lim_{x_{n+1} \rightarrow 0 } $ in the above formula, we obtain
\begin{align}
 &	\lim_{x_{n+1} \rightarrow 0 } \sum_{|\vec{d}_{\tilde{I}}| = d} \prod_{i \in \tilde{I}} \left(  \frac{1}{ \prod_{j \in \tilde{I}}(q^{d_i-d_j+1}x_{i j};q)_{d_{j}}} \cdot \frac{(x^{d_{i}}_{i}q^{\frac{d_{i}(d_{i}-1)}{2}})^{-|I|}}{ \prod_{j\in \tilde{I}^\complement}(qx_{i j};q)_{d_i}} \right)  \nonumber   \\
 =& \sum_{|\vec{d}_{\tilde{I}}| = d} \frac{(\prod_{i \in I }x_i^{d_i}q^{\frac{d_i(d_i-1)}{2}})^{-|I|}}{\prod_{i , j \in I }(q^{d_{ij}+1}x_{ij};q)_{d_j}\prod_{i\in I }\prod_{j \in \{[n] \backslash I \} }(qx_{ij};q)_{d_i}} \nonumber \\
 & \times \frac{(-1)^{|I| \cdot d_{n+1}}}{(q;q)_{d_{n+1}} q^{d_{n+1}(d-d_{n+1})} \prod_{i \in I } x^{d_{n+1}}_{i}} \nonumber \\
 =& \sum_{s=0}^d \frac{(-1)^{|I| \cdot s}}{(q;q)_{s} q^{s(d-s)} \prod_{i \in I } x^{s}_{i}} \times A_{d-s}\left( \vec{x} ,I, -|I| \right)   \nonumber
\end{align}
\end{proof}

\section{$K$-theoretic $I$-function with level structure}
\label{sec:5}
\subsection{Definitions}
\label{sec:6}
Let $X$ be a GIT quotient $V/\!/_\theta G$ where $V$ is a vector space and ${G}$ is a connected reductive complex Lie group. Let ${\mathcal{Q}}^{\epsilon}_{g,n}(X,\beta)$ be the moduli stack of $\epsilon$-stable quasimaps \cite{2013arXiv1304.7056C} parametrizing data $(C,p_1,...,p_n,\mathcal{P},s)$ where $C$ is an n-pointed genus $g$ Riemann surface, $\mathcal{P}$ is a principal $G$-bundle over $C$, $s$ is a section and $\beta \in \mathrm{Hom}(\mathrm{Pic}^G(V))$. There are natural maps:
\begin{align*}
e v_{i}: {\mathcal{Q}}^{\epsilon}_{g,n}(X,d) \rightarrow X, \quad i=1, \ldots, n	
\end{align*}
given by evaluation at the i-th marked point. There are line bundles
\begin{align*}
	L_{i} \rightarrow  {\mathcal{Q}}^{\epsilon}_{g,n}(X,d) , \quad i=1, \ldots, n
\end{align*}
called universal cotangent line bundles. The fiber of $L_i$ over the point $(C^{\epsilon},p_1,...,p_n,\mathcal{P},s)$ is the cotangent line to $C$ at the point $p_i$.

The permutation-equivariant $K$-theoretic quasimap invariants with level structures \cite{ruan2018level} are holomorphic characteristics over ${\mathcal{Q}}^{\epsilon}_{g,n}(X,d)$ of the sheaves:
\begin{align}
\left\langle \mathbf{t}(L), \ldots, \mathbf{t}(L) \right\rangle_{g, n, d}^{R,l,S_n,\epsilon}:=\pi_* \left({\mathcal{Q}}^{\epsilon}_{g,n}(X,d) ; \mathcal{O}_{g, n, d}^{v i r t} \otimes \prod_{m,i} L_{i}^{k}  t_{k,i} \mathrm{ev}_{i}^{*}\left(\phi_{i}\right) \otimes \mathcal{D}^{R,l} \right) \nonumber 
\end{align}
where $\mathcal{O}^{v i r}_{g,n,d}$ is called the virtual structure sheaf \cite{lee2004quantum}. And $\mathbf{t}(q)$ is defined as follows 
\begin{align*}
	\mathbf{t}(q)=\sum_{m \in \mathbb{Z}} t_{m} q^{m}, \quad t_{m}=\sum_{\alpha} t_{m, \alpha} w_{\alpha}
\end{align*}
where $\pi_*$ is the $K$-theoretic pushforward along the projection 
\begin{align*}
 \pi_* : [ {\mathcal{Q}}^{\epsilon}_{g,n}(X,d)/S_n ] \rightarrow [ pt ]
\end{align*}
and $\{\phi_\alpha \}$ is a basis in $K^0(X)\otimes Q$ and $t_{k,\alpha}$ are formal variables. The last term in (36)  is the level $l$ determinant line bundle over $\mathcal{Q}^{\epsilon}_{g,n}(X,d)$ defined as
\begin{align*}
	\mathcal{D}^{R,l} :=(\mathrm{det}R^{\bullet}\pi_*(\mathcal{P} \times _G R ))^{-l}
\end{align*}
the bundle $\mathcal{P} \times _G R $ is the pullback of the vector bundle $[V \times R / G]  \rightarrow [ V/G ]  $ along the evaluation map to the quotient stack $[V/G]$.

Similarly, we can define quasimap graph space $ \mathcal { QG }^{\epsilon} _ { 0 , n } ( X , \beta ) $ which parametrizes quasimaps with parametrized component $\mathbb{P}^1$, so there is a natural $\mathbb{C}^*$-action on quasimap graph space. Denoted by $\mathrm{F}_{0,\beta}$ the special fixed loci in $ ( \mathcal { QG }^{\epsilon} _ { 0 , n } ( X , \beta ) )^{\mathbb{C}^*}$, and denoted by $q$ the weight of cotangent bundle at $0 :=[1,0] $ of $\mathbb{P}^1$. for details, see \cite{2013arXiv1304.7056C}.
\begin{Definition}{\cite{ruan2018level}}
The permutation-equivariant $K$-theoretic $\mathcal{J}^{R,l,\epsilon}$-function of $V // {G}$ of level $l$ is defined as
\begin{align*}
\mathcal { J } _ { S _ { \infty } } ^ { R , l , \epsilon } ( \mathbf { t } ( q ) , Q ) 
 &:= \sum_{k \geq 0, \beta \in {\operatorname { Eff } ( V , \mathbf { G } , \theta )} } Q^{\beta} (ev_{\bullet})_{*} [\operatorname{Res}_{\operatorname{F}_{0,\beta}}( \mathcal{QG} _ { 0 , n } ^ { \epsilon } ( V / / \mathbf { G } , \beta )_{0})^{\mathrm{vir}} \otimes \mathcal { D } ^ { R , l } \otimes _ { i = 1 } ^ { n } \mathbf { t } ( L _ { i } ) ]^{S_n} \\
 &:= 1 + \frac { \mathbf { t } ( q ) } { 1 - q } +\sum _ { a } \sum _ { \beta \neq 0 } Q ^ { \beta } \chi \left( \operatorname{F} _ { 0 , \beta } , \mathcal { O } _ { \operatorname{F} _ { 0 , \beta } } ^ { \mathrm { vir } } \otimes { ev }_{\bullet} ^ { * } ( \phi _ { a } ) \otimes \left( \frac { \operatorname { t r } _ { \mathbb { C } ^ { * } } \mathcal { D } ^ { R , l } } { \lambda_{-1}^{\mathbb{C}^*}  N _ { \operatorname{F} _ { 0 , \beta } } ^ { \vee } } \right) \right) \phi ^ { a } \\
 &+ \sum_a \sum _ { n \geq 1 or \beta(L_\theta) \geq \frac{1}{\epsilon} \atop ( n , \beta ) \neq ( 1,0 ) } Q ^ { \beta } \left\langle \frac { \phi _ { a } } { ( 1 - q ) ( 1 - q L ) } , \mathbf { t } ( L ) , \ldots , \mathbf { t } ( L ) \right\rangle _ { 0 , n + 1 , \beta } ^ { R , l , \epsilon , S _ { n } } \phi ^ { a }
\end{align*}
where $\{ \phi_\alpha \}$ is a basis of $K^0(V/\!/{G})$ and $\{ \phi^\alpha \}$ is the dual basis with respect to twisted pairing  $( \ \ , \ \ )^{R,l}$ i.e.
\begin{align*}
	(u,v)^{R,l}:=\chi\left( X,u \otimes v \otimes \mathrm{det}^{-l}(V^{ss} \times_{G} R )  \right)
\end{align*}
\end{Definition}
\begin{Definition}{\cite{ruan2018level}}
When taking $\epsilon$ small enough, denoted by $\epsilon=0^{+}$, we call $\mathcal{J}^{R,l,0^{+}}(0)$ the small $I$-function of level $l$, i.e,
\begin{align*}
{I}^{R,l}(q;Q):= \mathcal { J } _ { S _ { \infty } } ^ { R , l , 0^{+} } ( 0 , Q ) =  1 + \sum _ { \beta \geq 0 } Q ^ { \beta } (ev_{\bullet})_{*} \left(  \mathcal { O } _ { \operatorname{F} _ { 0 , \beta } } ^ { \mathrm { vir } } \otimes  \left( \frac { \operatorname { t r } _ { \mathbb { C } ^ { * } } \mathcal { D } ^ { R , l } } { \lambda_{-1}^{\mathbb{C}^*}  N _ { \operatorname{F} _ { 0 , \beta } } ^ { \vee } } \right) \right) \cdot \mathrm{det}^l(V^{ss} \times_{G} R ) 
\end{align*}
\end{Definition}

\subsection{Level correspondence in Grassmann duality}
\label{sec:7}
Let $V$ be $r \times n$ matrixes $M_{r \times n}$, $G$ be the general linear group $GL_r$ and let $ \theta $ be the $\mathbf{det}: GL_r \rightarrow \mathbb{C}^*$, then we have
\begin{align*}
  	V/\!/_{\mathbf{det}}G = M_{r \times n}/\!/_{\mathbf{det}} G = Gr(r,n)
\end{align*}
 
  There is a natural $T=(\mathbb{C}^*)^n$-action $\mathbb{C}^n$ with weights $\mathbb{C}^n=\Lambda_1 + \cdots + \Lambda_n $, then deducing an action on $Gr(r,n)$ by $T \cdot A = AT$, $A \in  M_{r\times n}$. Using general abelian/non-abelian correspondence in \cite{2019arXiv190600775W} for $Gr(r,n)$, we have   
 \begin{align*}
  I^{Gr(r,n)}_{T} =& 1+ \sum _{d} \sum_{|\vec{d}|= d} \sum_{ \omega \in S_{r} / S_{r_{1}} \times \cdots \times S_{r_{h+1}}}	  \\
    &\omega \left[\frac{\prod_{1 \leqslant j<i \leqslant r} \prod_{ 1 \leqslant m \leqslant d_{i}-d_{j}}\left(1-L_{i} L_{j}^{-1} q^{m}\right)}{ \prod_{1 \leqslant i<j \leqslant r_{j} \atop 1 \leqslant m \leqslant d_{j}-d_{i}-1}\left(1-L_{i} L_{j}^{-1} q^{-m}\right) \prod_{1 \leqslant i<j \leqslant r}\left(1-L_{i}^{-1} L_{j}\right)}\prod_{i=1}^{r}\prod_{k=1}^{d_i}\prod_{m=1}^{n}\frac{1}{(1-q^kL_i\Lambda^{-1}_m)} \right] Q^d
  \end{align*}
 where $\vec{d}=\{d_1 \leq d_2 \leq \cdots \leq d_r \}$ such that $d_1 = d_2 = \cdots = d_{r_1} < d_{r_1+1}= \cdots = d_{r_1+r_2} < d_{r_1+\cdots+r_h}\cdots = d_{r_1+\cdots+r_h+r_{h+1}}$, i.e. $r_1+\cdots+r_{h+1}=r$. $\omega$ is the Weyl group acting on $L_i$ to change the index, $\{ L_i \}^r_{i=1}$ come from the filtration of tautological bundle $\mathcal{S}_r$ of $Gr(r,n)$. We could rewrite the equivariant $I$-function in the following way
 \begin{align}
  I^{Gr(r,n)}_{T} =& 1+ \sum _{d} \sum_{|\vec{d}|= d} \sum_{ \omega \in S_{r} / S_{r_{1}} \times \cdots \times S_{r_{h+1}}}\omega\left[\prod_{i,j=1}^r \frac{\prod^{d_i-d_j}_{k=-\infty}(1-q^kL_iL^{-1}_j)}{\prod^{0}_{k=-\infty}(1-q^kL_iL^{-1}_j)}  \prod_{i=1}^{r}\prod_{k=1}^{d_i}\prod_{m=1}^{n}\frac{1}{(1-q^kL_i\Lambda^{-1}_m)} \right] Q^d	\label{I-of-Gr}
 \end{align}
suppose $\omega$ changes $i_1$ to $i_2$ and $j_1$ to $j_2$, then one of the factors changes from 
\begin{align}
	 \frac{\prod^{d_{i_1}-d_{j_1}}_{k=-\infty}(1-q^kL_{i_1}L^{-1}_{j_1})}{\prod^{0}_{k=-\infty}(1-q^kL_{i_1}L^{-1}_{j_1})}\cdot \frac{\prod^{d_{i_2}-d_{j_2}}_{k=-\infty}(1-q^kL_{i_2}L^{-1}_{j_2})}{\prod^{0}_{k=-\infty}(1-q^kL_{i_2}L^{-1}_{j_2})} \label{unchange}
\end{align}
to 
\begin{align}
	 \frac{\prod^{d_{i_1}-d_{j_1}}_{k=-\infty}(1-q^kL_{i_2}L^{-1}_{j_2})}{\prod^{0}_{k=-\infty}(1-q^kL_{i_2}L^{-1}_{j_2})} \cdot \frac{\prod^{d_{i_2}-d_{j_2}}_{k=-\infty}(1-q^kL_{i_1}L^{-1}_{j_1})}{\prod^{0}_{k=-\infty}(1-q^kL_{i_1}L^{-1}_{j_1})} \label{changed}
\end{align}
since $\omega \in S_{r} / S_{r_{1}} \times \cdots \times S_{r_{h+1}}$, we have $d_{i_1} \neq d_{i_2}, d_{j_1} \neq d_{j_2}$. In (\ref{I-of-Gr}) we have an order of partition $\vec{d}$, one could see from (\ref{unchange}) to (\ref{changed}) that $\omega$-action is just rearrange $\{ d_i \}$ without changing the form. There is an unique $\omega \in S_{r} /\left(S_{r_{1}} \times \ldots \times S_{r_{h+1}}\right)$ whose inverse $\omega^{-1}$ arranges $\left(d_{1}, \ldots, d_{r}\right)$ in nondecreasing order $d_{1} \leq d_{2} \leq \ldots \leq d_{r}$ and then we have:
\begin{align}
I^{Gr(r,n)}_{T}=\sum_d \sum_{d_1+d_2+ \cdots + d_r = d} Q^d \prod_{i,j=1}^r \frac{\prod^{d_i-d_j}_{k=-\infty}(1-q^kL_iL^{-1}_j)}{\prod^{0}_{k=-\infty}(1-q^kL_iL^{-1}_j)}\prod_{i=1}^{r}\prod_{k=1}^{d_i}\prod_{m=1}^{n}\frac{1}{(1-q^kL_i\Lambda^{-1}_m)} \nonumber
\end{align}
note that in \cite{2011arXiv1110.3117T} where the author claimed a version of mirror theorem with a different $I$-function.

If we consider the standard representation of $GL_r$, denoted by $E_r$, then the associated bundle $\mathcal{P} \times _G R|_{F_{0,\beta}}$ can be identified with $\oplus_{i=1}^{r}L_i \otimes \mathcal{O}_{\mathbb{P}^1}(-d_i) $ 
\begin{align*}
	\mathrm{tr}_{\mathbb{C}^*} \mathcal{D}^{E_r,l}|_{F_{0,\beta}} =& \mathrm{tr}_{\mathbb{C}^*}\mathrm{det}^{-l}R^{\bullet}\pi_*( \oplus_{i=1}^{r} L_i \otimes \mathcal{O}_{\mathbb{P}^1}(-d_i) ) \\
	=& \mathrm{tr}_{\mathbb{C}^*} \mathrm{det}^{-l}( \oplus_{i=1}^{r} [L_i \otimes R^1\pi_*(\mathcal{O}_{\mathbb{P}^1}(-d_i))]^{-1}  ) \\
	=& \otimes_{i=1}^r \left( L_i^{d_i-1} \cdot q^{\frac{d_i(d_i-1)}{2}} \right)^l
\end{align*}
Similarly, if we take dual standard representation, denoted by $E_r^{\vee}$, then 
\begin{align*}
	\mathrm{tr}_{\mathbb{C}^*}\mathcal{D}^{E_r^\vee,l}|_{F_{0,\beta}} =& \mathrm{tr}_{\mathbb{C}^*}\mathrm{det}^{-l}( \oplus_{i=1}^{r} L^{-1}_i \otimes  R^{0}\pi_*(\mathcal{O}_{\mathbb{P}^1}(d_i)) ) \\
	=& \otimes_{i=1}^r \left( L_i^{d_i+1} \cdot q^{\frac{d_i(d_i+1)}{2}} \right)^l
\end{align*}
so the equivariant $I$-function of $Gr(r,n)$ with level structure is as follows
\begin{align}
I^{Gr(r,n),E_r,l}_{T,d}=\sum_{d_1+d_2+ \cdots + d_r = d} Q^d \prod_{i,j=1}^r \frac{\prod^{d_i-d_j}_{k=-\infty}(1-q^kL_iL^{-1}_j)}{\prod^{0}_{k=-\infty}(1-q^kL_iL^{-1}_j)}\prod_{i=1}^{r}\frac{(L^{d_{i}}_{i}q^{\frac{d_{i}(d_{i}-1)}{2}})^l}{\prod_{k=1}^{d_i}\prod_{m=1}^{n}(1-q^kL_i\Lambda^{-1}_m)} \label{I-1}
\end{align} 
and
\begin{align}
I^{Gr(r,n),E_r^\vee, l}_{T,d}=\sum_{d_1+d_2+ \cdots + d_r = d} Q^d \prod_{i,j=1}^r \frac{\prod^{d_i-d_j}_{k=-\infty}(1-q^kL_iL^{-1}_j)}{\prod^{0}_{k=-\infty}(1-q^kL_iL^{-1}_j)}\prod_{i=1}^{r}\frac{(L^{d_{i}}_{i}q^{\frac{d_{i}(d_{i}+1)}{2}})^l}{\prod_{k=1}^{d_i}\prod_{m=1}^{n}(1-q^kL_i \Lambda^{-1}_m)}	 \label{I-2}
\end{align}
\begin{Remark}
For the dual Grassmannian $Gr(n-r, n)$, here we still use the same presentation of GIT
quotient as in Grassmannian: $\mathrm{GL}(n-r, \mathbb{C})$ acts on $M_{(n-r) \times n}(\mathbb{C})$ by left matirx multiplication. $\left(\mathbb{C}^{*}\right)^{n}$-action on $\mathbb{C}^{n}$ is the dual action, so weights are $\mathbb{C}^{n}=\Lambda_{1}^{-1}+\cdots+\Lambda_{n}^{-1} .$ The action on $\operatorname{Gr}(n-r, n)$ is that $s \cdot B=B s^{-1},$ where $B \in M_{(n-r) \times n}(\mathbb{C})$ and $s \in T=\left(\mathbb{C}^{*}\right)^{n}$. So the corresponding equivariant $I$-function is as follows,
\begin{align}
I^{Gr(n-r,n),E_{n-r},l}_{T,d}=\sum_{d_1+d_2+ \cdots + d_{n-r} = d} Q^d \prod_{i,j=1}^{n-r} \frac{\prod^{d_i-d_j}_{k=-\infty}(1-q^k\tilde{L}_i\tilde{L}^{-1}_j)}{\prod^{0}_{k=-\infty}(1-q^k\tilde{L}_i\tilde{L}^{-1}_j)}\prod_{i=1}^{n-r}\frac{(\tilde{L}^{d_{i}}_{i}q^{\frac{d_{i}(d_{i}-1)}{2}})^l}{\prod_{k=1}^{d_i}\prod_{m=1}^{n}(1-q^k\tilde{L}_i\Lambda _m)} 	\nonumber
\end{align}
and
\begin{align}
	I^{Gr(n-r,n),E_{n-r}^\vee, l}_{T,d}=\sum_{d_1+d_2+ \cdots + d_{n-r} = d} Q^d \prod_{i,j=1}^{n-r} \frac{\prod^{d_i-d_j}_{k=-\infty}(1-q^k\tilde{L}_i\tilde{L}^{-1}_j)}{\prod^{0}_{k=-\infty}(1-q^k\tilde{L}_i\tilde{L}^{-1}_j)}\prod_{i=1}^{n-r}\frac{(\tilde{L}^{d_{i}}_{i}q^{\frac{d_{i}(d_{i}+1)}{2}})^l}{\prod_{k=1}^{d_i}\prod_{m=1}^{n}(1-q^k\tilde{L}_i \Lambda _m)} \nonumber
\end{align}
where $\tilde{L}_i$ for $i=1,\ldots,n-r$ come from the filtration of tautological bundle $\mathcal{S}_{n-r}$ over $Gr(n-r,n)$.	
\end{Remark}

Let $T$ act on Grassmannian $Gr(r,n)$ as before, then there are ${n \choose r} $ fixed pionts, i.e. denoted by $\{e_1, \ldots, e_n \}$, the basis of $\mathbb{C}^n$, then the subspace $V$ spanned by $\{e_{i_1}, \ldots, e_{i_r} \}$ is a $T$-fixed point . let 

\begin{align*}
\mathfrak{l}_{*}: {K}_{{T}}\left({Gr(r,n)}^{{T}}\right) \rightarrow {K}_{{T}}({Gr(r,n)})	
\end{align*}
the kernel and cokernel are $K_T(pt)$-modules and have some support in the torus $T$. From a very general localization theorem of Thomason \cite{thomason1992}, we know
\begin{align*}
\operatorname{supp} \text { Coker } \mathfrak{l}_{*} \subset \bigcup_{\mu}\left\{\mathfrak{t}^{\mu}=1\right\}	
\end{align*}
where the union over finitely many nontrivial characters $\mu$. The same is true of $\text { ker } \mathfrak{l}_{*}$, but since
\begin{align*}
{K}_{{T}}\left({Gr(r,n)}^{{T}}\right)={K}({Gr(r,n)}) \otimes_{\mathbb{Z}} {K}_{{T}}({pt})	
\end{align*}
has no such torsion, this forces $\text { ker } \mathfrak{l}_{*}=0$, so after inverting finitely many coefficients of the form $t^{\mu}-1$, we obtain an isomorphism, i.e. 
\begin{align*}
K_T^{loc}(Gr(r,n)^T) \cong K^{loc}_T(Gr(r,n))	
\end{align*}
we denote $K^{loc}_{T}(-)$ by
\begin{align*}
	K^{loc}_{T}(-) = K_T(-) \otimes_{R(T)} \mathcal{R}
\end{align*}
where $\mathcal{R} \cong \mathbb{Q}(t_1,\ldots,t_n)$ and $\{t_i\}$ are the charaters of torus $T$. 

Similarly, $T=\left(\mathbb{C}^{*}\right)^{n}$-action on $Gr(n-r,n)$ also has $\binom{n}{n-r} = \binom{n}{r}$ isolated fixed points, which is indexed by $(n-r)$-element subsets of $[n]$, so identification of $Gr(r,n)^{T}$ with $Gr(n-r,n)^{T}$ gives an $\mathcal{ R }$-module isomorphism of $K^{loc}_{T} ( Gr(r,n) )$ with $K^{loc}_{T} ( Gr(n-r,n) )$. Indeed, suppose $W$ is a subspace of dimension $r$ in a vector space $V$ of dimension $n$, then we have a natural short exact sequence
\begin{align*}
0 \rightarrow W \rightarrow V \rightarrow V / W \rightarrow 0
\end{align*}
taking the dual of this short exact sequence yields an inclusion of $(V / W)^{*}$ in $V^{*}$ with quotient $W^{*}$
\begin{align*}
0 \rightarrow(V / W)^{*} \rightarrow V^{*} \rightarrow W^{*} \rightarrow 0
\end{align*}
so $ \psi : W \mapsto (V/W)^{*}$ gives a cannocial equivariant isomorphism $Gr(r, V) \cong Gr(n-r, V^{*})$, where action of $T=\left(\mathbb{C}^{*}\right)^{n}$ on $V^{*}$ is induced from action of $T$ on $V$, thus, $\psi$ gives the canonical identification of fixed points
\begin{align} \label{idenFixPoints}
\psi : Gr(r,n)^{T} \longrightarrow Gr(n-r,n)^{T} \qquad <e_j>_{j \in I} \longmapsto <e^{j}>_{j \in I^\complement}
\end{align}
where $I$ is a set of $[n]$ with $|I|=r$, and $\{e^i\}_{i = 1} ^{n}$ is the dual basis of $\{e_i\}_{i = 1} ^{n}$. Now we can state the following Level correspondence in Grassmann duality
\begin{Theorem}(Level Correspondence)
For Grassmannian $Gr(r,n)$ and its dual Grassmannian $Gr(n-r,n)$ with standard $T=(\mathbb{C}^*)^n$ tours action, let $E_r$, $E_{n-r}$ be the standard representation of $\operatorname{GL}(r,\mathbb{C})$ and $\operatorname{GL}(n-r,\mathbb{C})$, respectively. Consider the following equivariant $I$-function
\begin{align*}
	I^{Gr(r,n),E_r,l}_T  =& 1 + \sum_{d=1}^\infty I^{Gr(r,n),E_r,l}_{T,d} Q^d ,\\
	I^{Gr(n-r,n),E_{n-r}^{\vee},-l}_T =& 1 + \sum_{d=1}^\infty I^{Gr(n-r,n),E^{\vee}_{n-r},-l}_{T,d} Q^d .	
\end{align*}
Then we have the following relations between $ I^{Gr(r,n),E_r,l}_{T,d} $ and $	I^{Gr(n-r,n),E^\vee_{n-r},-l}_{T,d} $ in $ K^{loc}_T(Gr(r,n)) \otimes \mathbb{C}(q) $ (which equals to $ K^{loc}_T(Gr(n-r,n)) \otimes \mathbb{C}(q) $):
\begin{itemize}
\item For $1-r\leq l\leq n-r-1$, we have 
\begin{align}
			I_{T,d}^{Gr(r,n),E_{r},l} = 	I_{T,d}^{Gr(n-r,n),E_{n-r}^\vee,-l}  \nonumber
\end{align}
\item For $l=n-r$, we have
\begin{align}
	   I_{T,d}^{Gr(r,n),E_{r},l}= \sum_{s=0}^d C_{s}( n-r,d)	I_{T,d-s}^{Gr(n-r,n),E_{n-r}^\vee,-l} \nonumber 
\end{align}	
where $C_{s}(k,d) $ is defined as
\begin{align*}
C_{s}(k,d)=\frac{(-1)^{k s}}{(q; q)_{s} q^{s(d-s+k)} \left(\bigwedge^{top}\mathcal{S}_{n-r} \right)^s   }	
\end{align*}
and $\mathcal{S}_{n-r}$ are the tautological bundle of $Gr(n-r,n)$
\item For $l=-r$, we have
\begin{align}
	 	I_{T,d}^{Gr(n-r,n),E_{n-r}^\vee,-l}= \sum_{s=0}^d D_{s}( r, d)	I_{T,d-s}^{Gr(r,n),E_{r},l} \nonumber
\end{align}
\begin{align*}
D_{s}(r,d)=\frac{(-1)^{rs}}{ (q;q)_sq^{s(d-s)}\left( \bigwedge^{top}\mathcal{S}_{r} \right)^s}
\end{align*}
and $\mathcal{S}_r$ are the tautological bundle of $Gr(r,n)$
\end{itemize}
\end{Theorem}
\begin{proof}
Form the discussion above, we prove the above identity by comparing $i^*_I I^{E_r, l}_{T}$ and $i^*_{I^\complement } I^{E_{n-r}^\vee, -l}_{T}$. Let $I= (j_1, \cdots , j _r)$ be the subset of $[n]=\{1, \ldots, n \}$, with $|I|=r$. Denote $v_1, v_2,\cdots, v_r$ the fiber coordinates in the fiber of $\mathcal{S}$ at fixed point $<e_j>_{j \in I}$, $\forall ( t_1, \cdots , t_n ) \in \left(\mathbb{C}^{*}\right)^{n}$, with weights $\mathbb{C}^n= \Lambda_1 + \cdots + \Lambda_n $ and
\begin{align*}
	&  ( t_1, \cdots , t_n ) \cdot (e_{j_1}, \cdots, e_{j_r} ;v_1, v_2,\cdots, v_r )  = (t_{j_1}e_{j_1}, \cdots, t_{j_r} e_{j_r} ;v_1, v_2,\cdots, v_r )
	\\ & \sim \operatorname{diag}(t_{j_1},\cdots, t_{j_r}) \cdot (t_{j_1} e_{j_1}, \cdots, t_{j_r} e_{j_r} ;v_1, v_2,\cdots, v_r )
	=  (e_{j_1}, \cdots, e_{j_r} ; t_{j_1} v_1, t_{j_2} v_2,\cdots, t_{j_r} v_r ) 
\end{align*}
so the weights of $i^*_{I}\mathcal{S}_r$ are $\{\Lambda_i \}_{i \in I}$ and the weights of  $i^*_{I^\complement}\mathcal{S}_{n-r}$ are $\{\Lambda^{-1}_i \}_{i \in I^\complement}$. Since the $I$-function is symmetric respect to $\{L_i\}$, then we could take any choice of weights
\begin{align}
    i^*_I I^{Gr(r,n),E_r, l}_{T,d} = \sum_{|\vec{d}_I| = d} \prod_{i,j \in I } \frac{\prod^{d_i-d_j}_{k=-\infty}(1-q^k \Lambda_i \Lambda^{-1}_j)}{\prod^{0}_{k=-\infty}(1-q^k \Lambda_i \Lambda^{-1}_j)}\prod_{i \in I}\frac{(\Lambda^{d_{i}}_{i}q^{\frac{d_{i}(d_{i}-1)}{2}})^l}{\prod_{k=1}^{d_i}\prod_{m \in [n]}(1-q^k \Lambda_i\Lambda^{-1}_m)}	\nonumber
\end{align}
and
\begin{align}
	i^*_{I^{\complement}} I^{Gr(n-r,n),E_{n-r}^\vee, -l}_{T,d}=\sum_{|\vec{d}_{I^{\complement}}| = d}  \prod_{i,j \in I^{\complement}} \frac{\prod^{d_i-d_j}_{k=-\infty}(1-q^k\Lambda^{-1}_i \Lambda_j)}{\prod^{0}_{k=-\infty}(1-q^k\Lambda^{-1}_i \Lambda_j)} \prod_{i \in I^{\complement}} \frac{(\Lambda^{-d_{i}}_{i}q^{\frac{d_{i}(d_{i}+1)}{2}})^{-l}}{\prod_{k=1}^{d_i}\prod_{m \in [n]}(1-q^k\Lambda^{-1}_i \Lambda_m)}	 \nonumber
\end{align}
using notation $\Lambda_{ij}=\Lambda_i \Lambda_j^{-1}$, and the following Lemma \ref{final-lemma} , we obtain
\begin{align}
	i^*_I I^{Gr(r,n),E_r, l}_{T,d} = \sum_{|\vec{d}_I| = d} \prod_{i \in I} \left(\frac{1}{   \prod_{j \in I}(q^{d_i-d_j+1}\Lambda_{i j};q)_{d_{j}}} \cdot \frac{(\Lambda^{d_{i}}_{i}q^{\frac{d_{i}(d_{i}-1)}{2}})^l}{ \prod_{j \in I^{\complement} }(q\Lambda_{i j};q)_{d_i}} \right) \label{I-I}
\end{align} 
and 
\begin{align}
	i^*_{I^{\complement}} I^{Gr(n-r,n),E_{n-r}^\vee,-l}_{T,d} =  \sum_{|\vec{d}_{I^{\complement}}| = d}  \prod_{i \in I^{\complement} } \left(   \frac{1}{\prod_{j \in I^{\complement} }(q^{d_i-d_j+1}\Lambda_{j i};q)_{d_{j}}} \cdot \frac{(\Lambda^{d_{i}}_{i}q^{\frac{-d_{i}(d_{i}+1)}{2}})^{l}}{ \prod_{j \in I }(q\Lambda_{j i};q)_{d_i}} \right) \label{I-J}
\end{align}
Comparing (\ref{I-I}) and (\ref{I-J}) with (\ref{Thm-A}) and (\ref{Thm-B}), we obtain the conclusion.
\end{proof}

\begin{Lemma} \label{final-lemma}
Let $I$ be the subset of $[n]=\{1, \ldots, n \}$. We have
\begin{align*}
    \prod_{i,j \in I} \left( \frac{\prod^{d_{ij}}_{k=-\infty}(1-q^k x_{ij} )}{\prod^{0}_{k=-\infty}(1-q^k x_{ij} )} \frac{1}{\prod_{k=1}^{d_i}(1-q^k x_{i j})} \right) = \prod_{i,j \in I}\frac{1}{(q^{d_{ij}+1} x_{i j};q)_{d_{j}}}
\end{align*}    
\end{Lemma}
\begin{proof}
It is sufficient to consider one term, if $d_i \geq d_j $, then 
\begin{align*}
	LHS = \frac{\prod_{k=1}^{d_{ij}}(1-q^k x_{ij})}{\prod_{k=1}^{d_i}(1-q^k x_{ij})} = \frac{1}{\prod_{k=d_{ij}+1}^{d_i}(1-q^k x_{ij})} =RHS
\end{align*}
If $d_i \leq d_j$, then
\begin{align*}
	LHS = \frac{1}{\prod_{k=d_{ij}+1}^{0}(1-q^k x_{ij}) \prod_{k=1}^{d_i}(1-q^k x_{ij})} =\frac{1}{\prod_{k=d_{ij}+1}^{d_i}(1-q^k x_{ij})} =RHS
\end{align*}
\end{proof}


\begin{thebibliography}{spmpsci}

\bibitem{2013arXiv1307.5997B}
Bonelli, Giulio and Sciarappa, Antonio and Tanzini, Alessandro and Vasko, Petr, Vortex partition functions, wall crossing and equivariant Gromov-Witten invariants, \href{https://ui.adsabs.harvard.edu/abs/2013arXiv1307.5997B}{arXiv1307.5997B} (2013)

\bibitem{givental2000wdvv}
Givental, Alexander, On the WDVV-equation in quantum K-theory, Michigan Mathematical Journal, Volume 48, 295-304 (2000)

\bibitem{lee2004quantum}
Y.P. Lee, Quantum K-theory I: Foundations, Duke Mathematical Journal, Volume 121, 389-424 (2004)

\bibitem{2015arXiv150903903G}
Givental, Alexander, Permutation-equivariant quantum K-theory V. Toric q-hypergeometric functions, \href{https://arxiv.org/abs/1509.03903}{arXiv:1509.03903} [math.AG] (2015)

\bibitem{2016arXiv160206494T}
{Tseng}, Hsian-Hua and {You}, Fenglong, K-theoretic quasimap invariants and their wall-crossing, \href{https://arxiv.org/abs/1602.06494}{arXiv:1602.06494} [math.AG] (2016)

\bibitem{2013arXiv1304.7056C}
{Ciocan-Fontanine}, Ionut and {Kim}, Bumsig, Wall-crossing in genus zero quasimap theory and mirror maps, Algebraic Geometry, Volume 1, 400-448 (2014)

\bibitem{ruan2018level}
Yongbin, Ruan and Ming, Zhang, The level structure in quantum K-theory and mock theta functions, \href{https://arxiv.org/abs/1804.06552}{arXiv:1804.06552} [math.AG] (2018)

\bibitem{2019arXiv190600775W}
Yaoxiong, Wen, K-Theoretic $I$-function of $V//_{\theta} \mathbf{G}$ and Application, \href{https://arxiv.org/abs/1906.00775}{arXiv:1906.00775} [math.AG] (2019)

\bibitem{thomason1992}
Thomason, R. W., Une formule de Lefschetz en $K$ -théorie équivariante algébrique, Duke Mathematical Journal, Volume 68, 447-462 (1992)

\bibitem{2018arXiv180406552R}
Yongbin, Ruan and Ming, Zhang, The level structure in quantum K-theory and mock theta functions, \href{https://arxiv.org/abs/1804.06552}{arXiv:1804.06552} [math.AG] (2018)

\bibitem{Nekrasov0901}
Nekrasov, N. and Shatashvili, S., Supersymmetric Vacua and Bethe Ansatz, Nuclear Physics B - Proceedings Supplements, Volume 192, 91-112 (2009) 

\bibitem{Nekrasov0908}
Nekrasov, N. and Shatashvili, S., Quantization of Integrable Systems and Four Dimensional Gauge Theories, XVIth International Congress on Mathematical Physics, 265-289 (2010)

\bibitem{Nekrasov9606}
Nekrasov, N., Four dimensional holomorphic theories, 348. Princeton University, 1996

\bibitem{jockers2019wilson}
Hans Jockers and Peter Mayr and Urmi Ninad and Alexander Tabler, Wilson loop algebras and quantum K-theory for Grassmannians, \href{https://arxiv.org/abs/1911.13286}{arXiv:1911.13286} [hep-th] (2019)

\bibitem{felder2018analyticity}
Felder, Giovanni and M{\"u}ller-Lennert, Martin, Analyticity of Nekrasov partition functions, Communications in Mathematical Physics, Volume 364(2), 683-718 (2018)

\bibitem{2011arXiv1110.3117T}
Taipale, Kaisa, K-theoretic J-functions of type A flag varieties, \href{https://arxiv.org/abs/1110.3117}{arXiv:1110.3117} [math.AG] (2011)

\bibitem{ueda20193d}
Ueda, Kazushi and Yoshida, Yutaka, 3d N= 2 Chern-Simons-matter theory, Bethe ansatz, and quantum K-theory of Grassmannians, \href{https://arxiv.org/abs/1912.03792}{arXiv:1912.03792} [hep-th] (2019)

\bibitem{jockers20183d}
Jockers, Hans and Mayr, Peter, A 3d gauge theory/quantum k-theory correspondence, Advances in Theoretical and Mathematical Physics, Volume 24, 327-457 (2020)

\bibitem{cite-key}
Bonelli, Giulio and Sciarappa, Antonio and Tanzini, Alessandro and Vasko, Petr, Vortex Partition Functions, Wall Crossing and Equivariant Gromov--Witten Invariants, Communications in Mathematical Physics, Volume 333(2), 717-760 (2015)
\end{thebibliography}
\end{document}